\documentclass[12pt]{amsart}
\usepackage{amsfonts,amssymb, amscd, latexsym, graphicx, psfrag, color, float}

\numberwithin{equation}{section}

\input xy
\xyoption{all}

\usepackage[bookmarks=true, bookmarksopen=true,%
bookmarksdepth=3,bookmarksopenlevel=2,%
colorlinks=true,%
linkcolor=blue,%
citecolor=blue,%
filecolor=blue,%
menucolor=blue,%
urlcolor=blue]{hyperref}

\newtheorem{dummy}{dummy}[section]
\newtheorem{assumptions}[dummy]{Assumptions}
\newtheorem{lemma}[dummy]{Lemma}
\newtheorem{theorem}[dummy]{Theorem}
\newtheorem{corollary}[dummy]{Corollary}
\newtheorem{proposition}[dummy]{Proposition}
\theoremstyle{definition}
\newtheorem{definition}[dummy]{Definition}
\newtheorem{example}[dummy]{Example}
\newtheorem{remark}[dummy]{Remark}

\newcommand{\bA}{\mathbb{A}}
\newcommand{\bC}{\mathbb{C}}

\newcommand{\bP}{\mathbb{P}}
\newcommand{\bQ}{\mathbb{Q}}
\newcommand{\bR}{\mathbb{R}}
\newcommand{\bZ}{\mathbb{Z}}

\newcommand{\cF}{\mathcal{F}}

\newcommand{\cM}{\mathcal{M}}

\newcommand{\cO}{\mathcal{O}}

\newcommand{\cS}{\mathcal{S}}
\newcommand{\cT}{\mathcal{T}}

\newcommand{\shF}{{\cF}}
\newcommand{\shM}{{\cM}}

\newcommand{\shO}{{\cO}}

\newcommand{\op}{\operatorname}

\newcommand{\X}{X}

\newcommand{\sympol}{\triangle}

\newcommand{\cPxfan}{{\Sigma^\vee}}

\renewcommand{\AA}{\mathbb{A}}

\newcommand{\CC}{\mathbb{C}}

\newcommand{\conv}{{\op{conv}}}

\newcommand{\dual}{{\vee}}

\newcommand{\gp}{{\op{gp}}}

\newcommand{\Hom}{\mathrm{Hom}}
\newcommand{\hra}{\hookrightarrow}

\renewcommand{\log}{{\op{log}}}
\newcommand{\lra}{\longrightarrow}

\newcommand{\Mor}{\op{Mor}}

\newcommand{\PP}{\mathbb{P}}
\newcommand{\Proj}{\mathrm{Proj}\,}

\newcommand{\ra}{\to}
\newcommand{\RR}{\mathbb{R}}

\newcommand{\Spec}{\mathrm{Spec}\,}

\newcommand{\tiM}{\widetilde{M}}
\newcommand{\tiN}{\widetilde{N}}

\newcommand{\triang}{\mathcal T}

\newcommand{\ZZ}{\mathbb{Z}}

\newcommand{\Maps}{\mathrm{Maps}}

\renewcommand{\subset}{\subseteq}

\newcommand{\PS}{\partial \sympol}

\newcommand{\WPSp}{\widehat{\partial\sympol'}}
\newcommand{\psprime}{\PS '}
\newcommand{\PSp}{\PS'}

\begin{document}

\title[Skeleta of Affine Hypersurfaces]{Skeleta of Affine Hypersurfaces}

\begin{abstract}
A smooth affine hypersurface $Z$ of complex dimension $n$
is homotopy equivalent to an $n$-dimensional cell complex.  
Given a defining polynomial $f$ for $Z$ as well as a regular triangulation $\cT_\sympol$
of its Newton polytope $\sympol,$ we provide
a purely combinatorial construction of a compact topological
space $S$ as a union of components of real dimension $n$, and
prove that $S$ embeds into $Z$
as a deformation retract.  In particular, $Z$ is homotopy equivalent to $S.$

\end{abstract}

\author{Helge Ruddat}
\author{Nicol\`o Sibilla}
\author{David Treumann}
\author{Eric Zaslow}
\maketitle

\setcounter{tocdepth}{1}
\tableofcontents

\section{Introduction}
The Lefschetz hyperplane theorem is equivalent to the assertion that a smooth affine variety $Z$ of complex dimension $n$ has vanishing homology
in degrees greater than $n$. A stronger version of this assertion is
attributed in the work of Andreotti-Frankel \cite{AF} to Thom:  $Z$ actually deformation retracts onto a cell complex of real dimension at most $n$.  
We will borrow terminology from symplectic geometry and call a deformation retract with this property a \emph{skeleton} for $Z$.  The purpose of this paper is to investigate the combinatorics of such skeleta for affine hypersurfaces $Z \subset \bC^{n+1}$, and a more general class of affine hypersurfaces in affine toric varieties.  For any such hypersurface, we give a combinatorial recipe for a large number of skeleta.

By ``combinatorial'' we mean that our skeleton makes contact with standard discrete structures from algebraic combinatorics, such as polytopes and partially ordered sets.  Before explaining what we mean in more detail, let us recall for contrast Thom's beautiful Morse-theoretic proof of Lefschetz's theorem, which provides a recipe of a different nature.  Fix an embedding $Z \subset \bC^N$, and let $\rho:Z \to \bR$ be the function that measures the distance to a fixed point $P \in \bC^N$.  For a generic choice of $P$, this is a Morse function, and since it
is plurisubharmonic, its critical points cannot have index larger than $n$.  Thom's skeleton is the union of stable manifolds for gradient flow of $\rho$.

This recipe reveals many important things about the skeleton (most important among them that the skeleton is \emph{Lagrangian}, a point that has motivated us but plays no role in this paper).  
The proof also works in the more general context of Stein
and Weinstein manifolds, see e.g. \cite{CE}.
However, finding an explicit description of these stable manifolds
requires one to solve some fairly formidable differential equations.
In this paper, we avoid this difficulty
by defining a skeleton through a simple, combinatorial
construction.

One might expect a rich combinatorial structure to emerge from the theory of Newton polytopes for hypersurfaces.  The situation is simplest
for hypersurfaces in $(\bC^*)^{n+1}$ rather than in $\bC^{n+1}$---we will explain this special case here, and the general situtation in Section \ref{subsec:intro2}.
If $Z$ is a hypersurface in $(\bC^*)^{n+1}$ we can write its defining equation as $f = 0$,
where $f$ is a Laurent polynomial of the form
\[
\sum_{m \in \bZ^{n+1}} a_m z^m.
\]
Here, if we write
$m$ as $(m_1,\ldots,m_{n+1})$ and the coordinates on $(\bC^*)^{n+1}$ as $z_1,\ldots, z_{n+1}$, then $z^m$ denotes the monomial $z_1^{m_1} \cdots z_{n+1}^{m_{n+1}}$.  The convex hull of the set of $m$ for which the coefficient $a_m$ is nonzero is called the \emph{Newton polytope} of $f$.  By multiplying $f$ by a monomial, we may assume without loss of generality that
the Newton polytope contains $0$.
The significance of this definition is that, for a generic choice of coefficients $a_m$, the topological type of the hypersurface depends only on this polytope.  
From this point of view, one goal might be to construct
a combinatorial skeleton which also depends only on this polytope. 
Actually, we need a triangulation too.  This should not be surprising:  a skeleton
is not unique, as different Morse functions will produce different skeleta.
Our combinatorial version of a Morse function turns out to be a
triangulation:  different triangulations will produce different skeleta.

\begin{definition}\label{def:skeldef}
Let $\sympol \subset \bR^{n+1}$ be a lattice polytope with $0\in \sympol.$  Let $\cT_\sympol$
be a star triangulation of $\sympol$ based at $0$, and define $\cT$ to be the set of simplices of
$\cT_\sympol$ not meeting $0.$  
Write $\partial\sympol'$ for the support of $\cT$.
(Note $\partial\sympol'$ equals the boundary $\partial\sympol$ if $0$ is an
interior point.  Note, too, that $\cT$ determines $\cT_\sympol$, even if $0\in \partial\sympol.$)
Define $S_{\sympol,\cT} \subset \PS' \times \Hom(\bZ^{n+1},S^1)$ to be the set of pairs $(x,\phi)$ satisfying the following condition:
\begin{quote}
$\phi(v) = 1$ whenever $v$ is a vertex of the smallest simplex $\tau \in \cT$ containing $x$
\end{quote}
\end{definition}

Put $S := S_{\sympol,\cT}.$  Then we have:

\begin{theorem}[{\bf Main Theorem}]
\label{thm:1.1}
Let $\sympol$ and $\cT$ be as in Definition \ref{def:skeldef}.  Let $Z$ be a generic smooth hypersurface whose Newton polytope is $\sympol$.  If $\cT$ is \emph{regular}, then $S$ embeds into $Z$ as a deformation retract.
\end{theorem}

The term ``regular'' is explained in Section \ref{sec:toricdegeneration}.  
We do not know if this hypothesis can be removed but note that every lattice polytope admits a regular lattice triangulation.  The role the triangulation
plays in the proof is in the construction of a degeneration of $Z$.  
Regularity of the triangulation allows the projection
$(x,\phi)\mapsto x$
of $S$ to $\PS$ (or to the support
of $\cT$ if $0$ is on the boundary of $\sympol$)
to be identified with the specialization map, under which
the skeleton of $Z$ projects to a
kind of nonnegative locus of toric components.  For more see Section \ref{proofgist} below.

\begin{figure}
\label{fig:quartikdual}
\centering
\includegraphics[angle=90,scale=.3]{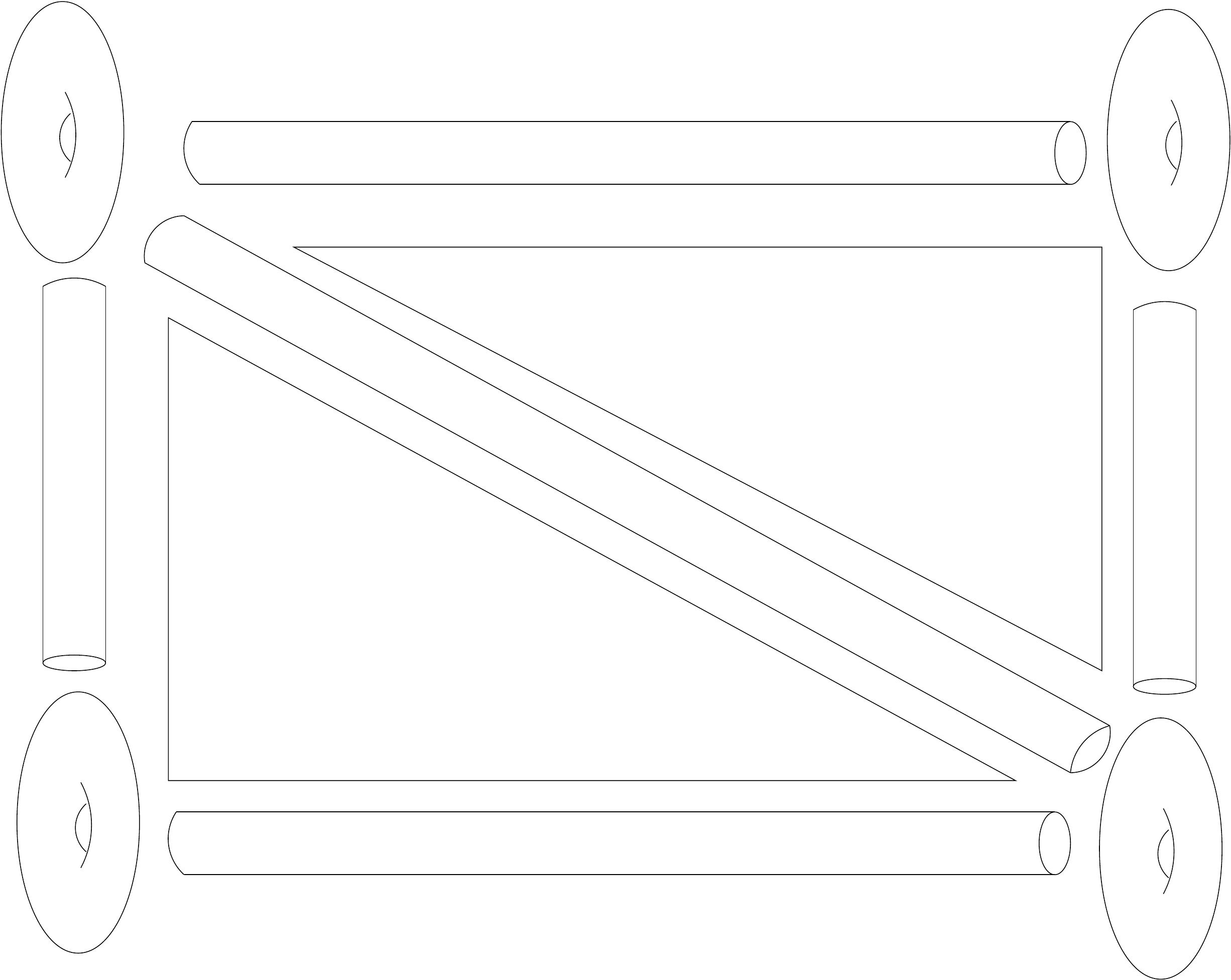}
\caption{\small The tetrahedron $\sympol \subset \bR^3$ with vertices at $(1,0,0)$, $(0,1,0)$, $(0,0,1)$, and $(-1,-1,-1)$ has a unique star triangulation $\cT_{\sympol}$.  The figure shows part of $S_{\sympol,\cT}$, which by Theorem \ref{thm:1.1} is a skeleton of a surface in $\bC^* \times \bC^* \times \bC^*$ cut out by the quartic equation $ax + by + cz + \frac{d}{xyz} + e = 0$
.  Each ``tube'' meeting one of the tori $S^1 \times S^1$ is attached along a different circle, and the resulting figure does not embed in $\bR^3$.  There is a sixth tube and two additional triangles ``behind'' the diagram, they are to be glued together in the shape of the tetrahedron $\sympol$.}
\end{figure}

\subsection{Hypersurfaces in affine toric varieties} 
\label{subsec:intro2}
In Section \ref{sec:five}, we prove an extension of our theorem to the case where
$Z$ is a smooth affine hypersurface in a more general
affine toric variety, such as $\bC^{n+1},$ $(\bC^*)^k\times \bC^l,$ or
even singular spaces such as $\bC^2/(\bZ/2).$
In these cases, we define the skeleton as a quotient of the
construction of Definition \ref{def:skeldef}.  

\begin{definition}
Let $\sympol$, $\cT$ and $S_{\sympol,\cT}$ be as in Definition \ref{def:skeldef}, so in particular $0\in\sympol$.  Let $K = \bR_{\geq 0} \sympol$ be the rational polyhedral cone generated by $\sympol$.  Define $S_{\sympol,\cT,K}$ to be the quotient of $S_{\sympol,\cT}$ by the equivalence relation
\[
(x,\phi) \sim (x',\phi') \text{ if $x = x'$ and $\phi\vert_{K_x \cap M} = \phi'\vert_{K_x \cap M}$}
\]
where $K_x$ denotes the smallest face of $K$ containing $x$.
\end{definition}

The cone $K$ determines an affine toric variety $\Spec(\bC[K \cap M])$.  If this is smooth or has at most one isolated singularity, and if $Z$ is a smooth hypersurface in $\Spec(\bC[K \cap M])$ with Newton polytope $\sympol$ and generic coefficients, then $Z$ deformation retracts onto a subspace homeomorphic to $S_{\sympol,\cT,K}$.  (The hypothesis that $K = \bR_{\geq 0} \sympol$ can be weakened; see Assumption \ref{assumptions}(3).)

To illustrate let us describe two skeleta of the subvariety of $\bC^3$ cut out by a generic quadric---in fact $x^2 + y^2 + z^2 = 1$ is sufficiently generic and we should expect $Z$ and its skeleton to be homotopy equivalent to a $2$-sphere.  In this case $\sympol$ is the convex hull of $\{(0,0,0), (2,0,0),(0,2,0), (0,0,2)\}$ and $K = \bR_{\geq 0}^3$.  The part $\partial \sympol'$ of $\partial \sympol$ to be triangulated is the face $\{(2,0,0),(0,2,0),(0,0,2)\}$, and we can describe $S_{\sympol,\cT,K}$ in terms of the projection map to $\partial \sympol'$.

{\bf Octahedron.}  If we give $\partial\sympol'$ its canonical triangulation, i.e. $\partial \sympol'$ itself is the only top-dimensional simplex, then the map $S_{\sympol,\cT,K}$ is homeomorphic to $S^2$.  In fact it is combinatorially an octahedron 
\[\{(x_1,x_2,x_3) \in \bR^3 \mid |x_1|+|x_2|+|x_3| = 2\}
\]
The map $S_{\sympol,\cT,K} \to \partial \sympol'$  is homeomorphic to the 8-to-1 map, branched over the boundary of $\partial \sympol'$, carrying a point $(x,y,z)$ to $(|x|,|y|,|z|)$.

{\bf Klein bottle sandwich.}  
If we triangulate $\partial \sympol'$ as in the diagram
\begin{figure}[H]
\includegraphics[height=1.5in]{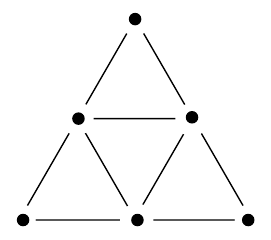}
\end{figure}
\noindent
the skeleton $S_{\sympol,\cT,K}$ is homeomorphic to the following space.  Let $X$ be a Klein bottle, i.e. a nontrivial $S^1$-bundle over $S^1$.  Let $\alpha$ and $\beta$ be two nonintersecting sections of the projection $X \to S^1$.  Then $S_{\sympol,\cT,K}$ is obtained from $X$ by attaching two copies of $\partial \sympol'$ along $\alpha$ and $\beta$.  Each attachment is along the boundary of the inner triangle of $\partial \sympol'$, i.e. the vertices $\{(1,1,0),(1,0,1),(0,1,1)\}$ and the edges joining them.

Since $\partial \sympol'$ is contractible, homotopically the attachments have the same effect of contracting $\alpha$ to a point and $\beta$ to a point.  This can be seen to be a $2$-sphere by identifying $X$ with the real (unoriented) blowup of $S^2$ at the north and south poles.

\subsection{Log geometry and the proof}
\label{proofgist}

The technique of the proof is to use the triangulation to construct
a degeneration of the ambient $(\bC^*)^{n+1},$ and with it the hypersurface.
Each component of the degeneration is an affine space $\bC^{n+1}$, or the quotient of an affine space by a finite commutative group, along which the hypersurface has a simple description: it is an affine Fermat hypersurface, and thus a finite branched cover over affine space $\bC^n$.
(see also the discussion of Mikhalkin's work in Section \ref{relatedwork}).
So the degenerated hypersurface is well understood.

\begin{example}
\label{twodex}
Consider the space $Z = \{-1 + x + y + x^{-1}y^{-1} = 0\}$ inside $\bC^*\times \bC^*,$
topologically a two-torus with three points removed.  The Newton polytope $\triangle =
\conv\{(1,0),(0,1),(-1,-1)\}$ $\subset \bR^2$ has a unique regular triangulation corresponding to
the unique lattice triangulation of its boundary.  To understand the associated degeneration,
first identify $\bC^*\times \bC^*$ with the locus $\{abc =1\}\subset \bC^3$ and
describe $Z$ by the equation $-1 + a + b + c = 0.$  Next we can identify this geometry
with the locus $t=1$ inside the \emph{family} $\{abc = t^3\}\subset \bC^4.$
At $t=0,$ we have for the
ambient space $\bC^2_{\{a=0\}}\cup \bC^2_{\{b=0\}}\cup \bC^2_{\{c=0\}},$
with the hypersurface described by $\{b+c=1\}\cup\{c+a=1\}\cup\{a+b=1\},$
i.e. a union of affine lines.
\end{example}

The degenerated hypersurface deformation retracts onto a simple locus which can be triangulated explicitly (this triangulation occurs for the first time in \cite{Deligne}).  In each component, the top-dimensional simplices of this triangulation are the \emph{nonnegative loci} of the components, together with their translates by a finite subgroup of $(\bC^*)^{n+1}$.  For instance in Example \ref{twodex}, the complex line $\{a+b = 1\} \subset \bC^2_{c = 0}$ retracts to the real interval $\{a+b = 1, a \geq 0, b \geq 0\})$.  What remains is to account for the topological difference between the degenerated hypersurface and the general one.  To the reader familiar with the theory of vanishing cycles (which measure the cohomological difference between the degenerate hypersurface and the general one), this will suggest that we take for a skeleton of $Z$ the preimage under a ``specialization'' map of the skeleton for $Z_0$.  Log geometry gives a way of making this precise.

The toric setting of log geometry is particularly simple.  A toric variety comes with a standard log structure which can be pulled back to a toric stratum, enabling the stratum to ``remember'' how it is embedded in the ambient space.  In short, the compact torus fixing the defining equations of a stratum of the degeneration serves as the expectional torus in a real, oriented blowup from which one can extract the nearby fiber of the degenerate hypersurface.

\begin{example}
To illustrate this point, consider first the local geometry of
the degeneration near a singular point of Example \ref{twodex}, i.e.
$\{uv = \epsilon^3\}\subset \bC^2.$  The two-torus
$S^1\times S^1\subset \bC^*\times \bC^*$ acts on $\bC^2$ and the
``antidiagonal'' circle fixes the defining equation for all $\epsilon.$
As $\epsilon$ goes to zero, this antidiagonal circle becomes
homotopically trivial since it retracts to the fixed point $(0,0)$ in $U=\{uv=0\}$.
We want to ``keep'' this circle by remembering the way we took the limit. 
Let $\rho:V\ra U$ be a retraction map from some fiber $\epsilon\neq 0$ to $U$.
To find the skeleton, we take the non-negative real locus $\bar S=\{(u,v)\in\RR_{\ge 0}^2|uv=0\}$ in $U$ and find the skeleton $S$ as the inverse image of $\bar S$ under $\rho$ which has the effect of attaching a circle at the point $(0,0)$ in $\bar S$.
Doing this globally in Example \ref{twodex} yields a skeleton $S$ that is the boundary of a triangle (by glueing the three $\bar S$'s) with a circle attached at each of the three vertices.
\end{example}

\subsection{Related Work} 
\label{relatedwork}

A skeleton for Fermat hypersurfaces was described by Deligne in \cite[pp. 88--90]{Deligne}, and this skeleton is visible in our own in a manner described in Remark \ref{rem:deligne}.
Our ``skeleta'' are different than the ``skeleta'' that appear in nonarchimedean geometry \cite{Berkovich,KS05}, but $\partial \sympol'$ plays a similar role in both constructions.  It would be interesting to study this resemblance further.

Hypersurfaces in algebraic tori have been studied by
Danilov-Khovanski\v{\i} \cite{DK} and Batyrev \cite{Ba93}.
Danilov-Khovankski\v{\i}
computed mixed Hodge numbers, while Batyev studied
the variation of mixed Hodge structures.
Log geometry has been extensively employed by Gross and Siebert \cite{GS2}
in their seminal work studying the degenerations appearing
in mirror symmetry.  Their strategy is crucial to our work, even though
we take a somewhat different track by working in a non-compact setting for hypersurfaces that are not necessarily Calabi-Yau.
The non-compactness allows us to deal with log-smooth log structures.
Mirror symmetry for general hypersurfaces was recently studied in \cite{GKR} (projective case) and \cite{AAK} (affine case) using polyhedral decompositions of the Newton polytope. This relates to the Gross-Siebert program by embedding the hypersurface in codimension two in the special fiber of a degenerating Calabi-Yau family. In this family, the hypersurface
coincides with the log singular locus --- see \cite{R} for the simplicial case.

In the symplectic-topological setting, Mikhalkin \cite{M}
constructed a degeneration of a projective algebraic hypersurface using a triangulation
of its Newton polytope to provide a higher-dimensional ``pair-of-pants''
decomposition.  He further identified a stratified torus fibration
over the spine of the corresponding amoeba.  This viewpoint was first
applied to homological mirror symmetry (``HMS'') by Abouzaid \cite{A}.
Mikhalkin's construction and perspective inform the current work greatly,
even though our route from HMS is a bit ``top-down.''  We describe it here.

When $\sympol$ is reflexive, $Z$ can be seen as the ``large volume limit"
of a family of Calabi-Yau hypersurfaces in the toric variety $\bP_\sympol$ defined by $\sympol.$  The
dual polytope $\sympol^\vee$ corresponds to the toric variety $\bP_{\sympol^\vee}$
containing the mirror family.  The mirror ``large complex structure limit'' $Z^\vee$
is the union of reduced toric divisors of $\bP_{\sympol^\vee}.$ 
In \cite{FLTZ} a relation was found between coherent sheaves on a toric variety,
such as $\bP_{\sympol^\vee},$ and a subcategory of
constructible sheaves on a real torus.  The subcategory is defined by a conical
Lagrangian $\Lambda$ in the cotangent bundle of the torus.  As discussed in \cite{TZ},
specializing to $Z^\vee$,
the complement of the open orbit of $\bP_\sympol$, can be achieved by
excising the zero section from $\Lambda$.  The resulting conical Lagrangian
is homotopy equivalent to the Legendrian $\Lambda^\infty$ at contact infinity of the cotangent bundle.
We can now explain how this relates to skeleta.  First,
when $\sympol$ is reflexive and simplicial and we choose $\cT$
to be the canonical triangulation
of its boundary, then $S$ is homeomorphic to $\Lambda^\infty.$  
In \cite{TZ} it is shown that $\Lambda^\infty$
supports a Kashiwara-Schapira sheaf of dg categories, and this
is equivalent to the ``constructible plumbing model'' of \cite{STZ}.  Following \cite{STZ},
this sheaf should be equivalent to perfect complexes on $Z^\vee$
and it is conjectured in \cite{TZ} 
that  under homological mirror symmetry
it is also equivalent to the sheaf of Fukaya categories,
conjectured to exist by Kontsevich, supported on the skeleton of $Z$.  
In particular, $S$ should be the skeleton of $Z$ itself, and in the simplicial
reflexive case this was conjectured in \cite{TZ}.

\subsection{Notation and conventions} 
\label{sec:conventions}

\subsubsection{Hypersurfaces in an algebraic torus}
\label{subsec:hypersurfaces}

Each $(m_0,...,m_n)\in\ZZ^{n+1}$ determines a monomial function $(\bC^*)^{n+1} \to \bC$ which we denote by 
$z^m = \prod_{i = 0}^{n+1} z_i^{m_i}$. 
If $f: (\bC^*)^{n+1} \to \bC$ is a Laurent polynomial we let $V(f) = \{z \mid f(z) = 0\}$ denote its zero locus.  The \emph{Newton polytope} of $f$ is the convex hull of the set of $m \in \bZ^{n+1}$ whose coefficient in $f$ is nonzero.  If the coefficients are chosen generically, then the diffeomorphism type of $V(f)$ depends only on the Newton polytope of $f$.  In fact it suffices that the extreme coefficients (i.e. the coefficients corresponding to the vertices of the Newton polytope) are chosen generically.  More precisely, 

\begin{proposition}[{e.g. \cite[Ch. 10, Cor. 1.7]{GKZ}}]
\label{prop:smoothn}
Let $A \subset \bZ^{n+1}$ be a finite set whose affine span is all of $\bZ^{n+1}$, and let $f_A$ be a Laurent polynomial of the form
\[
f(z) = \sum_{m \in A} a_m z^m
\]
There is a Zariski dense open subset $U_A \subset \bC^{|A|}$ such that, when the $(a_m)_{m \in A}$ are chosen from $U_A$, the variety $V(f_A)$ is smooth and its diffeomorphism type depends only on the convex hull of $A$.
\end{proposition}

\begin{remark} 
\label{remgeneric}
The precise condition that we mean by ``generic'' in Theorem \ref{thm:1.1} is as follows.
If $\overline Z$ denotes the closure of $Z$ in the projective toric variety $\PP_\sympol$ associated to $\sympol$
then we require $\overline Z\cap O$ to be either empty or smooth and reduced for each torus orbit $O\subset\PP_\Delta$. 
If this holds, $\overline Z$ is called $\sympol$-regular, a notion coined by Batyrev and Dwork, see \cite[Def. 3.3]{Ba93}.
Note that for each cell $\tau\in\cT$, we may consider the weighted projective space $\PP_\tau$ associated to $\tau$ 
and have a hypersurface $\overline Z_\tau\subset \PP_\tau$ given by the polynomial $f_\tau=\sum_{m\in\tau} a_m z^m$.
Now, we may state the precise definition of generic used in Thm~\ref{thm:1.1}:
we call $Z$ generic if $\overline Z$ is $\sympol$-regular and for each $\tau\in\cT$, $\overline Z_\tau$ is $\tau$-regular. 
The set of generic hypersurfaces forms an open subset of all hypersurfaces justifying the notion \emph{generic}.
\end{remark}

\subsubsection{Polytopes and triangulations}

An intersection of finitely many affine half-spaces in a finite-dimensional vector space is called \emph{polyhedron}. If it is compact, it is called \emph{polytope}. A polytope is the convex hull of its vertices.
Given a subset $A$ of a vector space, we denote its convex hull by $\conv A$.
Throughout, we let $M$ denote a free abelian group isomorphic to $\ZZ^{n+1}$ and set $M_\RR=M\otimes_\ZZ\RR \cong \RR^{n+1}$.
A polytope $\sympol \subset M_\bR$ is called a \emph{lattice polytope} if its vertices are in $M$.
We use the symbol $\subset$ for the face relation, e.g., $\tau\subset\sympol$ means that $\tau$ is a face of $\sympol$. 
The relative interior of a polytope $\tau$ will be denoted $\tau^\circ$.
Let $\PS$ denote the boundary of $\sympol$.  
A lattice triangulation $\cT_\sympol$ of a polytope $\Delta$ is a triangulation by lattice simplices.
Such a triangulation is called regular if there is a piecewise affine function convex function $h:\Delta\ra\RR$ such that the non-extendable closed domains where $h$ is affine linear coincide with the maximal simplices in $\cT_{\sympol}$.
We write $\cT_\sympol^{[0]}$ for the set of vertices of $\cT_\sympol$, and if $\tau$ is a simplex of $\cT_\sympol$ we write $\tau^{[0]}$ for the vertices of $\tau$.

\subsubsection{Monoids and affine toric varieties}
\label{toricgeomandmonoids}

We denote by $\Spec R$ the spectrum of a commutative ring $R$.  When $R$ is a noetherian commutative algebra over $\bC$. We will often abuse notation by using the same symbol $\Spec R$ for the associated complex analytic space and $\cO$ for $\cO^{\mathrm{an}}$.  Given $f_1,...,f_r\in R$, we write $V(f_1,...,f_r)$ for the subvariety of $\Spec R$ defined by the equations $f_1=...=f_r=0$.

A \emph{monoid} is a set with an associative binary operation that has a unit and a two-sided identity.  For us, all monoids will be commutative.  Given a monoid $\mathsf{M}$ with an action on a set $V$, we write $\mathsf{M} T$ for the orbit of a subset $T\subseteq V$. We often use this when $V$ is an $\RR$-vector space, $T$ some subset and $\mathsf{M}=\RR_{\ge0}$ the non-negative reals.  Further notation for monoids is discussed in Section \ref{subsec:logdefs}.

By a \emph{cone} $\sigma \subset M_\bR$ we shall always mean a rational polyhedral cone, i.e. a set of the form
\[
\{ \sum_{i \in I} \lambda_i v_i \mid \lambda_i \in \bR_{\geq 0} \}
\]
where $\{v_i\}_{i \in I}$ is a finite subset of lattice vectors in $M_\bR$.  A cone is called \emph{strictly convex} if it contains no nonzero linear subspace of $M_\bR$.  Gordon's Lemma \cite[p. 12]{Fu} states that the monoid $M \cap \sigma$ is finitely generated.  The monoid ring $\bC[M \cap \sigma]$ is then noetherian.  For $m \in M \cap \sigma$ we write $z^m$ for the corresponding basis element of $\bC[M \cap \sigma]$; it can be regarded as a \emph{regular monomial} function $\Spec \bC[M \cap \sigma] \to \bC$.

We have the following standard device for describing points on an affine toric variety.
If $x$ is a point of $\Spec \bC[M \cap \sigma]$, write $\mathrm{ev}_x:M \cap \sigma \to \bC$ for the map
\[
\mathrm{ev}_x(m) = z^m \text{ evaluated at $x$}
\]
Each $\mathrm{ev}_x$ is a homomorphism of monoids from $M \cap \sigma$ to $(\bC,\times)$.   The universal property of the monoid ring gives the following
\begin{proposition}
\label{prop:specpoints}
Let $\sigma$ be a rational polyhedral cone in $M_\bR$.  Then $x \mapsto \mathrm{ev}_x$ is a one-to-one correspondence between the complex points of $\Spec \bC[M \cap \sigma]$ and the monoid homomorphisms $M \cap \sigma \to \bC$.
\end{proposition}

\subsection*{Acknowledgments}  
We thank Gabriel Kerr for his help with the case of nonreflexive polytopes. 
We thank Nir Avni, Johan de Jong, Grigory Mikhalkin, Sam Payne and Bernd Siebert for helpful discussions.
The work of HR is supported by DFG-SFB-TR-45 and the Carl Zeiss Foundation.  The work of DT is supported by NSF-DMS-1206520.  The work of EZ is supported by NSF-DMS-1104779 and by a Simons Foundation Fellowship.

\section{Degenerations of hypersurfaces}
\label{sec:toricdegeneration}

We fix a lattice polytope $\sympol \subset M_\bR$ with $0\in \sympol$. 
Let $K\subset M_\bR$ be a convex subset.
A continuous function
$h:K\ra\RR$ is called \emph{convex} if for each $m,m'\in K$ and we have
\[
\frac{h(m) + h(m')}{2} \geq h\left(\frac{m + m'}{2}\right).
\]
We fix a lattice triangulation $\cT_\sympol$ of $\sympol$ with the following property:
$0\in \cT_{\sympol}^{[0]}$ and there exists a convex piecewise linear function $h:\bR_{\geq 0}\sympol\rightarrow \bR$
taking non-negative integral values on $M$ such that
the maximal dimensional simplices in $\cT_\sympol$ coincide with the non-extendable closed domains of linearity of $h|_\sympol$. 
We also choose such a function, $h$.
Triangulations with this property are often called \emph{regular} or \emph{coherent}.  Every lattice polytope
containing the origin
supports a regular lattice triangulation.  Since $h$ is linear on the $(n+1)$-simplices of $\cT_{\sympol}$, this triangulation is ``star-shaped with center $0$'' in the sense that each simplex in $\cT_{\sympol}$ is contained in $\PS$ or contains the origin $0$. 
We define the triangulation $\cT$ by 
\[
\cT=\{\tau\in\cT_\sympol\;|\;\tau\subset\PS,0\not\in\tau\},
\]
i.e., the set of simplices of $\sympol$ not containing the origin.  We denote the union of all $\tau \in \cT$ by $|\cT|$, and sometimes by $\PS'.$
Since $\cT$ induces $\cT_\sympol$,
we call $\cT$ regular if the induced $\cT_\sympol$ is regular.

We fix a Laurent polynomial $f\in\CC[M]$ of the form
\begin{equation}
\label{eq:f}
f = a_0 + \sum_{m \in \cT^{[0]}} a_m z^m.
\end{equation}
We suppose that all coefficients are real, that $a_0 < 0$, that $a_m >0$ for $m \in \cT^{[0]}$, and that they are chosen generically with this property.  
We write $V(f) \subset \Spec\CC[M]$ for the hypersurface in the algebraic torus defined by $f = 0$.

\begin{remark}
\label{rem:AZarDense}
Since the positivity conditions on the $a_m$ are Zariski dense, it follows by Proposition \ref{prop:smoothn} that ${V}(f)$ is smooth and diffeomorphic to any generic hypersurface whose Newton polytope is $\sympol$.
\end{remark}

Using the piecewise linear function $h$, we can give a toric degeneration of $(\CC^*)^{n+1}$ and an induced degeneration of $V(f)$ in the style of Mumford.
We construct this degeneration in Sections \ref{sec:31} and \ref{sec:32}.

\begin{remark}
\label{rem:originboundary}

In case the origin is on the boundary of $\sympol$, it is natural to embed $V(f)$ into the following partial compactification of $(\bC^*)^{n+1}$.  The polytope $\sympol$ generates a cone $\bR_{\geq 0} \sympol \subset M_\bR$.  The cone is not usually strictly convex, e.g. if $0 \in \sympol^\circ$ then this cone is all of $M_\bR$.  In any case, $f$ is always a linear combination of monomials in $\bR_{\geq 0} \sympol \cap M$ and defines a hypersurface in $\Spec \bC[M \cap \bR_{\geq 0} \sympol]$ which we denote by $\overline{V}(f)$.  If $0 \in \sympol^\circ, \overline{V}(f) = V(f).$
\end{remark}

\subsection{Degeneration of the ambient space}
\label{sec:31}

The total space of the degeneration will be an affine toric variety $Y$ closely related to the affine cone over the projective toric variety whose moment polytope is $\sympol$.  More precisely, it is an affine subset of the affine cone over a blowup of this toric variety.  The construction makes use of the \emph{overgraph cone} in $M_\bR \oplus \bR$, coming from the piecewise-linear function $h$.

\subsubsection{The overgraph cone}
\label{subsubsec:Gamma}

Let $\Sigma_\cT$ be the fan in $M_\bR$ whose nonzero cones are the cones over the simplices in $\cT$, i.e.,
\[
\Sigma_\cT =\{ \RR_{\ge0}\tau | \tau\in\cT \}.
\]
When $0$ is an interior lattice point,  $\Sigma_\cT$ is a complete fan. 
In general its support is the cone $\bR_{\geq 0} \sympol$.

Since $\cT$ is regular, $\Sigma_{\cT}$ is projected from part of the boundary of a rational polyhedral cone in $M_\bR \oplus \bR$.  We fix such a cone and call it the \emph{overgraph cone}.  Let us define it more precisely.
Set $\tiM=M\oplus\ZZ$ and $\tiM_{\bR}=\tiM\otimes_\ZZ\RR$.
The overgraph cone of $h$
is defined to be
\[
\Gamma_{\geq h} = \{(m,r) \in \tiM_{\bR} \mid m \in \bR_{\geq 0}\sympol,\; r \geq h(m)\}
\]

Each cone in $\Sigma_\cT$ is isomorphic to a proper face of $\Gamma_{\geq h}$ under the projection $\tiM_\bR \to M_\bR$.  The inverse isomorphism is given by $m \mapsto (m,h(m))$.  Since $h$ takes integral values on $M$, the faces of $\Gamma_{\geq h}$ that appear in this way form a rational polyhedral fan in $\tiM$.  We record this observation in the following lemma:

\begin{lemma}
\label{lem:Gammahtau}
Let $\bR_{\geq 0} \tau$ be a cone in $\Sigma_{\cT}$ and let $\Gamma_{\geq h,\tau} \subset \Gamma_{\geq h}$ be the face
\[
\Gamma_{\geq h,\tau} = \{(m,h(m)) \in \Gamma_{\geq h} \mid m \in \bR_{\geq 0} \tau\}. 
\]
Then the projection $\Gamma_{\geq h,\tau} \to \bR_{\geq 0}\tau$ is an isomorphism of cones inducing and isomorphism of monoids
$\Gamma_{\geq h,\tau} \cap \tiM \to \bR_{\geq 0} \tau \cap M$.
\end{lemma}

\subsubsection{Degeneration}
\label{subsubsec:Y}
The overgraph cone determines an affine toric variety that we denote by $Y$, i.e.
\[
Y = \Spec \bC[\Gamma_{\geq h} \cap \tiM]
\]
Define $\pi:Y \to \AA^1$ to be the map given by the regular monomial function $t=z^{(0,1)}$ on $Y$.  Let $Y_0 \subset Y$ denote the fiber $\pi^{-1}(0)$.  Since $t$ is a monomial, $Y_0$ is torus invariant in $Y$, but in general has many irreducible components.  Let us call the components of $\pi^{-1}(0)$ the \emph{vertical divisors} of the map $\pi$ and then call the remaining toric prime divisors \emph{horizontal divisors}.

\begin{remark}
\label{rem:specpoints}
Since $Y$ is an affine toric variety, we can identify the points of $Y$ (by Proposition \ref{prop:specpoints}) with the space of monoid homomorphisms $(\tiM \cap \Gamma_{\geq h},+) \to (\bC,\times)$.  In this description, $Y_0$ is the subset of monoid homomorphisms $\phi:\tiM \cap \Gamma_{\geq h} \to \bC$ carrying $(0,1)$ to $0$.
\end{remark}

\begin{proposition}
\label{prop:ambientdegeneration}
The map $\pi:Y \to \AA^1$ has the following properties:
\begin{enumerate}
\item $\pi^{-1}(\bC^*)=\Spec \CC[(\RR_{\ge0}\sympol)\cap M]\times\CC^*$ and the restriction of $\pi$ to $\pi^{-1}(\bC^*)$ is the projection onto the second factor.

\item The subscheme structure on $Y_0 = \pi^{-1}(0)$ is reduced.
\item $\pi$ is a toric degeneration of $\Spec \CC[(\RR_{\ge0}\sympol)\cap M]$. The restriction of $\pi$ to the complement of the union of horizontal divisors is a degeneration of $\Spec \CC[M]\cong(\CC^*)^{n+1}$.
\end{enumerate}
\end{proposition}
\begin{proof}
Localizing to $\pi^{-1}(\bC^*)$ means adjoining $t^{-1}$ to the ring $\bC[\Gamma_{\geq h} \cap \tiM]$ which yields \linebreak 
$\bC[((\RR_{\ge0}\sympol)+\RR(0,1)) \cap \tiM]=\bC[(\RR_{\ge0}\sympol)\cap M]\otimes_\CC\CC[\ZZ]$. This gives the first statement in (3) as well as (1). 

To prove (2), note that since $h$ takes integral values on $M$, 
any element $m\in\Gamma_{\geq h}\cap \tiM$ can uniquely be written as
\[
m' + k (0,1)
\]
with $k\in\ZZ_{\geq 0}$ and $ m'$ in $\Gamma_{\geq h,\tau}\cap \tiM$ for some $\tau\in\cT$.
The second statement in (3) is best seen in the fan picture. If $\Sigma$ is the normal fan of $\Gamma_{\geq h}$, removing the horizontal divisors amounts to restricting to the subfan 
$\Sigma'\subseteq\Sigma$ of cones that have no rays contained in $(0,1)^\perp$. The map $\pi$ is given by mapping $\Sigma'$ to the fan $\{\{0\},\RR_{\ge 0}\}$ and $\{0\}\in\Sigma'$ is the only cone that maps to $\{0\}$, so the general fiber is indeed an algebraic torus.
\end{proof}

Let us describe the vertical 
and the horizontal 
divisors in more detail.  

\begin{proposition}
\label{prop:vhdivisors}
Let $Y$ and $\pi$ be as above and for each $\tau \in \cT$ let $\Gamma_{\geq h,\tau}$ be as in Lemma \ref{lem:Gammahtau}.
\begin{enumerate}
\item The assignment $\tau \mapsto \Spec \bC[\Gamma_{\geq h,\tau}]$ is a bijection between the vertical divisors of $\pi$ and the $n$-dimensional simplices of $\cT$.
\item The assignment 
\[
\tau \mapsto \Spec\CC[(\RR_{\ge 0}\{(m,h(m))|m\in\tau\} +\RR_{\ge 0}(0,1))\cap \tiM]
\]
is a bijection between the horizontal divisors of $\pi$ and the $n$-dimensional simplices $\tau$ of $\cT_{\sympol}$ with $0 \in \tau$ and $\tau\subset\partial\sympol$.
\end{enumerate}
\end{proposition}

\begin{proof}
The toric prime divisors in $Y$ correspond to the codimension one faces of $\Gamma_{\ge h}$. 
Such a face corresponds to a vertical divisor if and only if it contains $(0,1)$. This implies (1) and (2).
\end{proof}

\begin{example}
\label{ex:degeneration}
For a simple illustrative example, take $\sympol = \mathrm{conv}\{(0,0),(2,0),(0,2)\},$ with lattice points 
named as follows:
\[
\begin{array}{ccccccc}
\bullet & \bullet & \bullet & \bullet & \bullet \\
\bullet &e & \bullet & \bullet & \bullet \\
\bullet &c & d & \bullet & \bullet\\
\bullet &0 & a & b & \bullet \\
\bullet & \bullet & \bullet & \bullet & \bullet 
\end{array}
\]
Then $\partial \sympol'$ is
the line segment between $b$ and $e,$ and let us take $\cT$ to be the fine triangulation with maximal
simplices $\overline{bd}$ and $\overline{de}$.  Then $\Sigma_\cT$ is supported in the first
quadrant, and its maximal cones are generated by $\{b,d\}$ and $\{d,e\}.$  Let $h$ be the piecewise linear
function supported on $\Sigma_\cT$ with $$h(a) = 1,\quad h(c) = 1,\quad h(d) = 1.$$
Then the ring $\bC[\Gamma_{\geq h}\cap \widetilde{M}]$ of regular functions on $Y$ can be identified with
$\bC[a,c,d,t]/(ac-dt).$ (We identify $a$ with $z^at^{h(a)}$ and so on.)
For $\lambda\neq 0$, the fiber $\pi^{-1}(\lambda)$ can be identified with $\bC^2$ via $(a,c) \leftrightarrow (a,c,ac/\lambda,\lambda),$
as it must from Proposition \ref{prop:ambientdegeneration} (1) since $\Spec \bC[(\bR_{\geq 0}\sympol)\cap M] \cong \bC^2.$
 
Setting $t=0$ gives $Y_0$ as $\bC[a,c,d]/ac$ which reveals the vertical divisors
as $V(c,t) \cong \bC^2 = \{(a,d)\}$ and $V(a,t)\cong \bC^2 = \{(c,d)\}.$
The horizontal divisors are $V(a,d)$ and $V(c,d)$, which can also be identified with $\bC^2.$
In Example \ref{ex:hypersurface} we shall return to this example to consider hypersurface degenerations
when we have a polynomial $f$ with $\mathrm{Newt}(f)=\sympol.$
\end{example}

\begin{remark}
In these examples we have used coordinates on $Y$ indexed by lattice points in $\sympol$.  This is always possible for $n \leq 1$, but for larger $n$ the coordinate ring of $Y$ can require many more generators.
\end{remark}

\subsubsection{Orbit closures in $Y_0$}
\label{sec:orbitclosures}
For each $\tau \in \cT$ let $Y_{0,\tau}$ be the $(\dim(\tau)+1)$-dimensional affine toric variety
\[
Y_{0,\tau} = \Spec(\bC[\tiM \cap \Gamma_{\geq h,\tau}])
\]
where $\Gamma_{\geq h,\tau}$ is defined in Lemma \ref{lem:Gammahtau}.  Since $\Gamma_{\geq h,\tau}$ is a face of $\Gamma_{\geq h}$, $Y_{0,\tau}$ is a torus orbit closure in $Y$.  
Each vertical divisor of $t$ is of the form $Y_{0,\tau}$ where $\tau$ is an $n$-dimensional simplex of $\cT$ by Proposition~\ref{prop:vhdivisors}.

Restricting regular functions from  $Y$ to $Y_{0,\tau}$ induces the ring quotient map
\[
\bC[\Gamma_{\geq h} \cap \tiM] \to \bC[\Gamma_{\geq h,\tau} \cap \tiM]
\]
whose kernel is the ideal generated by monomials $z^{(m,r)}$ with $(m,r) \notin \Gamma_{\geq h,\tau}$.  By Lemma \ref{lem:Gammahtau}, we may identify $\bC[\Gamma_{\geq h,\tau} \cap \tiM]$ with $\bC[\bR_{\geq 0} \tau \cap M]$.  

\subsubsection{Projection onto $\bP^{\dim(\tau)}$}

The action of $\Hom(M,\bC^*)$ on $Y_{0,\tau}$ factors through an action of the quotient torus $\Hom(\bR\tau \cap M,\bC^*)$.   We now define a finite subgroup $D_\tau \subset \Hom(\bR \tau \cap M,\bC^*)$ which will play an important role for us.

\begin{definition}
\label{def:Dtau}
Let $D_\tau$ be the finite commutative group
\[
D_\tau = \Hom((M\cap \bR \tau)\big/ \bZ \tau^{[0]},\bC^*) 
\]
We regard $D_\tau$ as a subgroup of $\Hom(M\cap \bR \tau,\bC^*)$, and let it act on the coordinate ring of $Y_{0,\tau} = \Spec \bC[M \cap \bR_{\geq 0} \tau]$ by
\[
d . z^m = d(m)z^m
\]
\end{definition}

\begin{proposition}
\label{prop:Dtauinvariant}
The invariant subring
\[
\bC[ M \cap \bR_{\geq 0} \tau ]^{D_\tau} \subset \bC[ M \cap \bR_{\geq 0} \tau ]
\]
is the monoid ring $\bC[\bZ_{\geq 0} \tau^{[0]}]$.  In other words, it is a polynomial ring whose $\dim(\tau)+1$ variables are parameterized by the vertices $\tau^{[0]}$ of $\tau$.
\end{proposition}

\begin{proof}
The monomials $z^m$ for $M \cap \bR_{\geq 0} \tau$ form a basis of eigenvectors for the $D_\tau$-action on $\bC[M \cap \bR_{\geq 0} \tau]$.  The invariants are therefore generated by those monomials $z^m$ for which $d(m) = 1$ for all $d \in D_\tau$.  Each vertex of $\tau$ has this property, and thus
\[
\bZ_{\geq 0} \tau^{[0]} \subset  \{m \in \bR_{\geq 0} \tau \cap M \mid d(m) = 1 \text{ for all $d \in D_\tau$}\}
\]
Let us show the containment is an equality, i.e. that for each $m \in M \cap \bR_{\geq 0}\tau$, if the monomial $z^m$ is $D_\tau$-invariant then $m$ is a $\bZ_{\geq 0}$-linear combination of the vertices of $\tau$.  This
follows from the fact that $\tau^{[0]}$ is a basis 
for the vector space $\bR\tau$, and that each element of $M \cap \bR_{\geq 0} \tau$ 
can be written in this basis with coefficients in $\bQ_{\geq 0}$.
Indeed, let
$v_0,v_1,\ldots,v_{\dim(\tau)}$ be the vertices of $\tau$ and for $i = 0,\ldots,\dim(\tau)$ define $d_i$ by  $d_i(v_j) = \delta_{i,j}.$  Suppose that $z^m$ is a $D_\tau$-invariant monomial.
Then since $m = \sum a_i v_i$ where each $a_i$ is in $\bQ_{\geq 0}$, we have $d_j(m) = e^{2\pi i a_j} = 1$ for all $j,$ i.e. $a_j\in \bZ_{\geq 0}.$
\end{proof}

\begin{proposition}
\label{prop:Y0taubasepoint}
The fiber of the $D_\tau$-quotient map
\[
Y_{0,\tau} \to \bC^{\dim(\tau) + 1}
\]
above $0 \in \bC^{\dim(\tau) + 1}$ is a single point.
\end{proposition}

Note there is a mild abuse of notation here: the coordinates of $\bC^{\dim(\tau)+1}$ are not indexed by the integers $1,\ldots,n+1$ but the vertices of $\tau$.

\begin{proof}
We use the description of Proposition \ref{prop:specpoints}.  The origin in $\bC^{\dim(\tau)+1}$ corresponds to the monoid homomorphism $\bZ_{\geq 0} \tau^{[0]} \to \bC$ that  carries each vertex of $\tau$ (and in fact each nonzero element of $\bZ_{\geq 0} \tau^{[0]}$) to $0 \in \bC$.  To prove the Proposition, it suffices to show that this extends to a monoid map $M \cap \bR_{\geq 0} \tau \to \bC$ in a unique way.   Indeed, this is the map that carries $0$ to $1$ and each nonzero element of $M \cap \bR_{\geq 0} \tau$ to $0$.
\end{proof}

Since the $D_\tau$-invariant ring $\bC[\bZ_{\geq 0} \tau^{[0]}]$ is a polynomial ring, we may endow it with a grading by declaring that $\deg(z^m) = 1$ whenever $m$ is a vertex of $\tau$.

\begin{definition}
\label{def:projYtoPtau}
Let $0 \in Y_{0,\tau}$ and $0 \in \Spec(\bC[\bZ_{\geq 0} \tau^{[0]}])$ denote the points of Proposition \ref{prop:Y0taubasepoint}.  We define a space $\bP^{\dim(\tau)}$ and a map $\pi_{\tau}:Y_{0,\tau} \setminus \{0\} \to \bP^{\dim(\tau)}$ as follows:
\begin{enumerate}
\item We let $\bP^{\dim(\tau)} = \Proj(\bC[\bZ_{\geq 0} \tau^{[0]})$, where the grading on the coordinate ring is indicated above.  In other words, $\bP^{\dim(\tau)}$ is a projective space whose homogeneous coordinates are naturally indexed by the vertices of $\tau$.
\item
We let $q_{\tau}:Y_{0,\tau} \setminus \{0\} \to \bP^{\dim(\tau)}$ denote the composite map
\[
Y_{0,\tau} \setminus \{0\} \to \bC^{\dim(\tau)+1} \setminus \{0\} \to \bP^{\dim(\tau)}
\]
where the first map is the $D_\tau$-quotient map of Proposition \ref{prop:Dtauinvariant} and the second map is the tautological map.
\end{enumerate}
\end{definition}
Note the abuse of notation in (1): if $\dim(\tau) = \dim(\tau')$ we will usually regard $\bP^{\dim(\tau)}$ as different from $\bP^{\dim(\tau')}$.

\subsection{Degeneration of the hypersurface}
\label{sec:32}

In Proposition \ref{prop:ambientdegeneration}, we have seen 
that the general fiber of $\pi:Y \to \AA^1$ is isomorphic to $\Spec\CC[M \cap \RR_{\ge0}\sympol]$.  We now describe a degeneration of $\overline{V}(f) \subset \Spec\CC[(\RR_{\ge0}\sympol)\cap M]$ contained 
in the family $\pi:Y\rightarrow \AA^1$. The total space of the degeneration is the hypersurface in $Y$ cut out by a regular function $\tilde{f}$ on $Y.$ On the open orbit of $Y$, $\tilde{f}$ looks like
\[
\tilde{f} = a_0 + \sum_{m\in \cT^{[0]}} a_m  z^{(m,h(m))} = a_0 + \sum_{m \in \cT^{[0]}} a_m z^m t^{h(m)}
\]
where the $a_m$ are the same coefficients as in $f$ (Equation \ref{eq:f}).  Denote the vanishing locus of $\tilde{f}$ by $X = V(\tilde{f})$.

\begin{remark}
\label{rem:originboundary2}
When $0$ is in the interior of $\sympol$, $X$ is a degeneration of $V(f)$. When $0$ is on the boundary, $X$ is a degeneration of $\overline{V}(f) \supset V(f)$ defined in Remark \ref{rem:originboundary}.
\end{remark}

\begin{example}
\label{ex:hypersurface}
 We return to the setting of Example \ref{ex:degeneration} to study the associated degeneration of
 the smooth hypersurface defined by the polynomial $f = -1 + x^2 + xy + y^2.$
 Note that $\Spec\bC[(\bR_{\geq 0}\sympol)\cap M]\cong \bC^2$, so we will degenerate both
 $\overline V(f)$ and inside it $Z = V(f) \subset \bC^*\times \bC^*.$  In Example \ref{ex:c*c*skeleton}
 we shall study the skeleton of $Z$, and in Example \ref{ex:ccskeleton} we will turn to investigate  
 the skeleton of $\overline{V}(f)$ in $\bC^2$.

The function $\tilde f: Y\rightarrow \bC$ is written $\tilde f = -1 + b + d + e$.
Recalling from Example 2.7 that $b = z^{(2,0)}t^{h(2,0)} = x^2t^2$, $d = z^{(1,1)}t^{h(1,1)} = xyt$, and $ e = z^{(0,2)}t^{h(0,2)} = y^2t^2$, we see that $\tilde f$ specializes to $f$ on $\pi^{-1}(1)$.
\end{example}

The restriction of $\tilde{f}$ to $Y_{0,\tau}$ is the image of $\tilde{f}$ under the ring quotient map
\[
\bC[\Gamma_{\geq h} \cap \tiM] \to \bC[\Gamma_{\geq h,\tau} \cap \tiM]
\]
that carries $z^{(m,r)}$ to itself if $(m,r) \in \Gamma_{\geq h,\tau}$ and to $0$ otherwise.  In other words, $\tilde{f}\vert_{Y_{0,\tau}}$ is given by
\[
\tilde{f}\vert_{Y_{0,\tau}} = a_0 + \sum_{m \in \tau^{[0]}} a_m z^{(m,h(m))}
\]
Let us denote the image of $\tilde{f}\vert_{Y_{0,\tau}}$ under the identification $Y_{0,\tau}=\Spec\bC[\bR \tau \cap M]$ by $f_\tau$.  We record this in the following definition:

\begin{definition}
\label{def:ftau}
Let $a_m$ be the coefficients of $f$ (Equation \ref{eq:f}).
\begin{enumerate}
\item
Let $f_\tau \in \bC[\bR_{\geq 0} \tau \cap M]$ denote the expression
\[
f_\tau = a_0 + \sum_{m \in \tau^{[0]}} a_m z^m
\]
regarded as a regular function on $Y_{0,\tau}$.  Let $X_{0,\tau}$ be the hypersurface in $Y_{0,\tau}$ cut out by $f_\tau$.
\item
Let $\ell_\tau \in \bC[\bZ_{\geq 0} \tau^{[0]}]$ denote the expression
\[
\ell_\tau = \sum_{m \in \tau^{[0]}} a_m z^m
\]
regarded as a homogeneous linear function on $\bP^{\dim(\tau)}$.    Let $V(\ell_\tau) \subset \bP^{\dim(\tau)}$ denote the hyperplane cut out by $\ell_\tau$.
\end{enumerate}
\end{definition}

\begin{proposition}
\label{prop:degenaffine}
Fix $\tau\in\cT$ and denote by
$p_\tau : X_{0,\tau}\ra \PP^{\dim(\tau)}$ the composition \[
X_{0,\tau} \hookrightarrow Y_{0,\tau} \setminus \{0\} \stackrel{q_{\tau}}\to \bP^{\dim(\tau)}
\]
where the second map is the projection of Definition \ref{def:projYtoPtau}.  Then
\begin{enumerate}
\item $p_\tau$ is a finite proper surjection onto the affine space $\bP^{\dim(\tau)} \setminus V(\ell_\tau) \cong \bC^{\dim(\tau)}$.  
\item $p_\tau$ induces an isomorphism 
\[X_{0,\tau}/D_\tau \cong \bP^{\dim(\tau)} \setminus V(\ell_\tau)\]
where $D_\tau$ is as in Definition \ref{def:Dtau}.
\item The ramification locus of $p_\tau$ is contained in the coordinate hyperplanes of $\bP^{\dim(\tau)}$.
\end{enumerate}
\end{proposition}

\begin{proof}
The following implicit assertions of the Proposition are trivial to verify:
\begin{itemize}
\item since $a_0 \neq 0$, the point $0 \in Y_{0,\tau}$ of Proposition \ref{prop:Y0taubasepoint} does not lie on $X_{0,\tau}$.
\item since the monomials that appear in $f_\tau$ belong to $\bZ_{\geq 0} \tau^{[0]}$, they are invariant under the action of $D_\tau$.   In particular $X_{0,\tau}$ is invariant under 
$D_\tau$.
\end{itemize}
Note that (1) is a consequence of (2).  Since $a_0 \neq 0$ the function $f_\tau = a_0 + \ell_\tau$ cannot vanish anywhere that $\ell_\tau$ vanishes.  Therefore the image of $p_{\tau}$ is contained in $\bP^{\dim(\tau)} \setminus V(\ell_\tau)$.  To complete the proof of (2), let us show that the affine coordinate ring $R_1$ of $\bP^{\dim(\tau)} \setminus V(\ell_\tau)$ is the $D_\tau$-invariant subring of the affine coordinate ring $R_2$ of $X_{0,\tau}$.  We have
\[
\begin{array}{ccc}
R_1 & = & \bC[\bZ_{\geq 0} \tau^{[0]}] / (a_0 + \ell_\tau) \\
R_2 & = & \bC[M \cap \bR_{\geq 0} \tau] / (f_\tau)
\end{array}
\]
and the short exact sequences
\[
\begin{array}{c}
\xymatrix{
0 \ar[r] & \bC[\bZ_{\geq 0}\tau^{[0]}] \ar[r]^{a_0 + \ell_\tau} & \bC[\bZ_{\geq 0}\tau^{[0]}] \ar[r] & R_1 \ar[r] & 0
} \\
\xymatrix{
0 \ar[r] & \bC[M \cap \bR_{\geq 0} \tau] \ar[r]^{f_\tau} & \bC[M \cap \bR_{\geq 0} \tau] \ar[r] & R_2 \ar[r] & 0 
}
\end{array}
\]

Part (2) of the Proposition is now a consequence of the observation that taking $D_\tau$-invariants preserves exact sequences, and that $\bC[\bZ_{\geq 0} \tau] = \bC[M \cap \bR_{\geq 0} \tau]^{D_\tau}$ by Proposition \ref{prop:Dtauinvariant}.

Now let us prove (3).  Let $H \subset \bP^{\dim(\tau)}$ be the union of coordinate hyperplanes.  By (2), to show that $p_\tau$ is unramified away from $H$ it suffices to show that $D_\tau$ acts freely on $X_{0,\tau}$ away from $p_\tau^{-1}(H)$.  In fact $D_\tau$ acts freely on $Y_{0,\tau} \setminus q_\tau^{-1}(H)$.  This completes the proof.

\end{proof}

\subsection{Degeneration of the compact hypersurface}
\label{sec:tordegcompact}

The families $\pi:Y \to \AA^1$ and $\pi:X \to \AA^1$ of Sections \ref{sec:31} and \ref{sec:32} have fairly natural algebraic relative compactifications (i.e., ``properifications'' of the maps $\pi$) that we review here.

We define the polyhedron
$$\overline\Gamma = \{(m,r)\in\tiM_\RR\mid m\in\sympol, r\ge h(m)\}$$
which is contained in $\Gamma_{\geq h}$.
We set $\tiN=\Hom(\tiM,\ZZ)$, $\tiN_\RR=\tiN\otimes_\ZZ\RR$.
The \emph{normal fan} of $\overline\Gamma$ is the fan 
$\Sigma_{\overline\Gamma}=\{ \sigma_\tau\mid \tau\subset\overline\Gamma \}$
where
$\sigma_\tau=\{n\in \tiN_\RR\mid \langle m-m',n\rangle\ge 0\text{ for all }m\in\overline\Gamma,m'\in\tau\}$ and $\langle\cdot,\cdot\rangle:\tiM\otimes \tiN\ra\ZZ$ is the natural pairing.
Let $\overline Y$ denote the toric variety associated to $\Sigma_{\overline\Gamma}$.
It is covered by the set of affine open charts of the shape $\Spec\CC[\sigma_\tau^\dual\cap \tiM]$ where $\tau\in\overline\Gamma^{[0]}$ and 
$$\sigma_\tau^\dual=\RR_{\ge 0} \{m-m'\mid m\in\overline\Gamma,m'\in\tau\}\subset \tiM_\RR$$
is the dual cone of $\sigma_\tau$.
Note that $\sigma^\dual_0=\Gamma_{\geq h}$, so we have an open embedding $Y\subseteq \overline Y$.
Since $(0,1)\in\sigma^\dual_\tau$ for all $\tau\subset\overline\Gamma$, $\pi$ extends to a regular function
$$\pi:\overline Y\ra \AA^1.$$
The support of $\Sigma_{\overline\Gamma}$ is $\{(n,r)\in\tiN_\RR\mid r\ge 0\}$ and pairing with the monomial $(0,1)$ sends this to $\RR_{\ge0}$. Thus by the Proposition in \S2.4 of \cite{Fu}, we have
\begin{lemma}$\pi:\overline Y\ra \AA^1$ is proper.\end{lemma}
\begin{corollary} Let $\overline{X}$ denote the closure of $X$ in $\overline Y$.  Then $\pi:\overline X\ra\AA^1$, the restriction of $\pi$ to $\overline X$, is proper.
\end{corollary}

\section{The skeleton}
\label{sec:two}

\subsection{Definition of the skeleton}
\label{sec:skel}
We adopt the notation from \S\ref{sec:toricdegeneration}, in particular $M\cong \bZ^{n+1}$ is a lattice and $\sympol \subset
M_\bR = M\otimes_\bZ \bR$ a lattice polytope containing $0$ and with
regular lattice triangulation $\cT$ of $\PS'$. 
For $x \in \PS'$, let us denote by $\tau_x$ the lowest-dimensional simplex of $\cT$ containing $x$.

\begin{definition}
\label{def:skeleton}
With $\sympol$, $\cT$, and $x \mapsto \tau_x$ given as above,
define the topological subspace 
\[
\begin{array}{ccc}
S_{\sympol,\triang}\subset\PS' \times \Hom(M,S^1)
\end{array}
\] 
to be the set of pairs $(x,\phi)$ satisfying
\[
\phi(v) = 1 \in S^1 \text{ whenever $v \in M$ is a vertex of $\tau_x$}.
\]
\end{definition}

The fibers of the projection $S_{\sympol,\triang} \to \PS$ are constant above the interior of each simplex of $\cT$.  In fact these fibers are naturally identified with a subgroup of the torus $\Hom(M,S^1)$.  Let us introduce some notation for these fibers:

\begin{definition}
\label{def:Gtau}
For each simplex $\tau \in \cT,$ let $G_\tau$ denote the commutative group contained in the torus 
$\Hom(M,S^1)$ given by
\[
G_\tau := \{\phi \in \Hom(M,S^1) \mid \phi(v) = 1 \text{ whenever $v \in M$ is a vertex of $\tau$}\}
\]
We denote the identity component of $G_\tau$ by $A_\tau$ and the discrete quotient $G_\tau/A_\tau = \pi_0(G_\tau)$ by $D_\tau$. 
That is, we have the short exact sequence of abelian groups
\begin{equation} \label{groupexseq}
1\to A_\tau\to G_\tau\to D_\tau \to 1.
\end{equation}
\end{definition}
This sequence can also be obtained by applying the exact contravariant functor $\Hom(\cdot,S^1)$ to the sequence
\begin{equation*} \label{groupexseq2}
0\leftarrow M\big/((\RR\tau)\cap M)\leftarrow M\big/(\ZZ\tau^{[0]})\leftarrow ((\RR\tau)\cap M)\big/(\ZZ\tau^{[0]}) \leftarrow 0.
\end{equation*}
On finite groups, $\Hom(-,S^1) = \Hom(-,\bC^*)$, so the definition of $D_\tau$ given here agrees with Definition \ref{def:Dtau}.
Here are two additional properties of the groups $G_\tau$:
\begin{enumerate}
\item $A_\tau$ is a compact torus of dimension $n - \dim(\tau)$.
\item When $\tau' \subset \tau$, there is a reverse containment $G_\tau \subset G_{\tau'}$.
\end{enumerate}

\begin{remark}
The fiber of $S_{\sympol,\triang} \to \PS'$ above $x$ is connected if any only if $D_{\tau_x}$ is trivial, so if and only if the simplex $\conv(\{0\} \cup \tau_x)$ is unimodular.  A triangulation whose simplices are unimodular uses every lattice point of $\sympol$ as a vertex, but the converse is not true.  For instance, $\tau$ might contain a triangle of the form $\{(1,0,0),\,(0,1,0),\,(1,1,N)\}$ for $N > 1$.
\end{remark}

\begin{remark}
\label{rem:quotientofS}
Define an equivalence relation on $S_{\sympol,\cT}$ by setting $x \sim y$ if both of the following hold:
\begin{itemize}
\item $x$ and $y$ project to the same element of $\PS'$,
\item $x$ and $y$ are in the same connected component of the fiber of this projection.
\end{itemize}
If $\cT$ is unimodular, then the quotient $S_{\sympol,\cT}/\!\!\sim$ is just $\PS'$.
In general $\partial\sympol'$ is some branched
cover of $\PS',$ with stratum $\tau^\circ$ having
covering group $D_\tau$.
We may write it as a regular cell complex which we denote by $\widehat{\PS'}$, i.e. 
\[
\widehat{\PS'} := S_{\sympol,\cT}/\!\!\sim \;\cong \bigcup_{\tau \in \cT} \tau^{\circ} \times D_\tau.
\]
We investigate this in more detail in the Section \ref{sec:regularcell}.
\end{remark}

\begin{example}
 \label{ex:c*c*skeleton}
 Picking up from Example \ref{ex:hypersurface}, we consider $Z = V(f)$ in $\bC^*\times \bC^*$ and compute its skeleton.
 We write $\phi \in \Hom(\bZ^2,\bR/\bZ)$ as $\phi = (\alpha,\beta),$ where $\phi(u,v) = \alpha u + \beta v \mod \bZ.$ 
 The vertex $\{b\} = \{(2,0)\} \in \cT$ has $G_{\{b\}} = \{(\alpha,\beta) \mid 2\alpha \equiv 0\} \cong \bZ/2 \times \bR/\bZ$ ---
 namely $\alpha$ is $0$ or $1/2$ and $\beta$ is free ---
 which is homeomorphic to two disjoint circles.  Similarly, $G_{\{e\}}$ is two disjoint circles:  $\alpha$ is free and $\beta$ is $0$ or $1/2.$
 $G_{\overline{bd}} = G_{\overline{de}} \cong \bZ/2$ is two points:  $(\alpha,\beta) = (0,0)$ or $(1/2,1/2).$
 $G_{\{d\}}$ is a single circle, $\beta = -\alpha$, since $d$ is primitive.  Up to homotopy, the fibers over the edges
 serve to attach the circles over $b$ and $e$ to $G_{\{d\}}$, meaning $S_{\sympol,\cT}$ is homotopic to a bouquet of five circles. A schematic representation of $S_{\sympol,\cT}$ is given below
 \begin{figure}[H]
 \label{fig:gorgeous}
\includegraphics[height=1.7in]{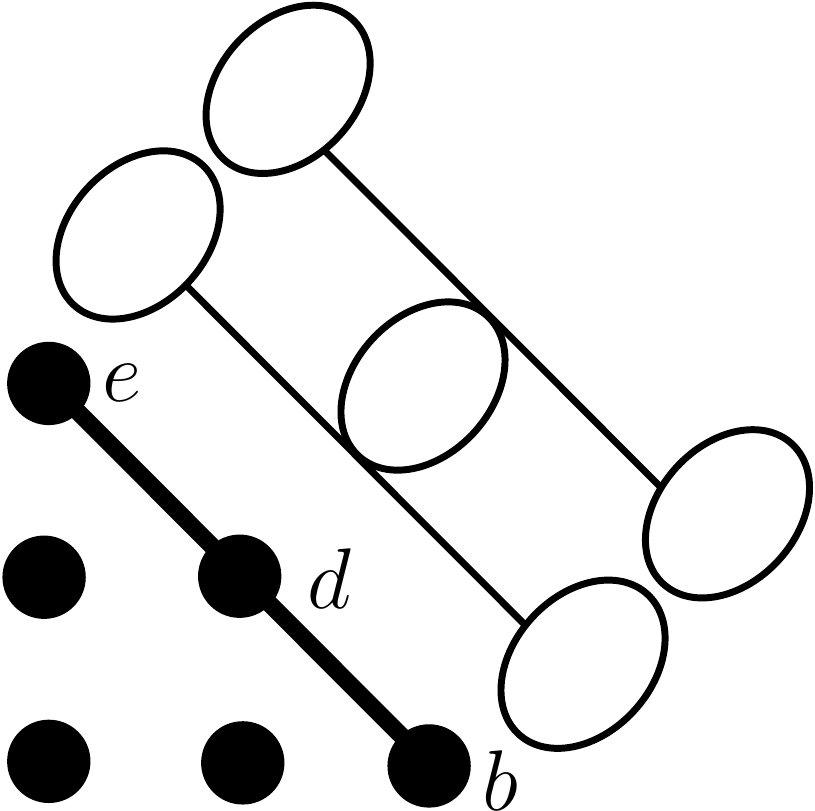}
\end{figure}
\noindent
 We shall see that $Z$ is homotopy equivalent to $S_{\sympol,\cT}$
 after investigating the skeleton of $\overline{V}(f)$ in Example \ref{ex:ccskeleton}
 in Section \ref{sec:five}.
\end{example}

\begin{remark}
The vertices of the triangulation $\cT$
generate the rays of a (stacky) fan $\cPxfan \subset M_\bR$.  It is shown in \cite{FLTZ,FLTZ2} that coherent sheaves on the toric Deligne-Mumford stack associated with $\cPxfan$ can be regarded as constructible sheaves on a compact torus with singular support in a conic Lagrangian $\Lambda_{\cPxfan} \subset N_\bR/N \times M_\bR \cong T^* (N_\bR/N)$.  This ``coherent-constructible correspondence'' is a full embedding of triangulated categories --- conjecturally an equivalence.
The conic Lagrangian $\Lambda_{\cPxfan}$ is noncompact.  Its Legendrian ``boundary'' $\Lambda_{\Sigma^\vee}^\infty$
at contact infinity of $T^* (N_\bR/N)$ is  homeomorphic to $S_{\sympol,\cT}$ --- see also
Section \ref{relatedwork}.
\end{remark}

\subsection{$\widehat {\PS'}$ as a regular cell complex}
\label{sec:regularcell}

Let us describe the combinatorics of $\widehat {\PS'}$ in some more detail.

\begin{definition}
\label{def:That}
For each $\tau \in \cT$ let $D_\tau$ be the finite commutative group given in Definition \ref{def:Dtau}.  We define the partially ordered set $\widehat{\cT}$ as follows.
\begin{enumerate}
\item If $\tau,\tau' \in \cT$ have $\tau \subset \tau'$, define a homomorphism $\mathrm{res}_{\tau',\tau}: D_{\tau'} \to D_\tau$ by the following formula.  If $d:\bR \tau' \cap M \to S^1$ is an element of $D_{\tau'}$, then $\mathrm{res}_{\tau',\tau}(d):\bR \tau \cap M \to S^1$ is given by
\[
\mathrm{res}_{\tau',\tau}(d)(m) = d(m)
\]
\item Let $\widehat{\cT}$ denote the set of pairs $(\tau,d)$ where $\tau \in \cT$ and $d \in D_\tau$.  We regard $\widehat{\cT}$ as a partially ordered set with partial order given by
\[
(\tau,d) \leq (\tau',d') \text{ whenever $\tau \subset \tau'$ and $\mathrm{res}_{\tau',\tau}(d') = d$} 
\]
\end{enumerate}
\end{definition}

Each $(\tau,d) \in \widehat{\cT}$ determines a map
\[
i_{\tau,d}:\tau \to \widehat{\PS'}
\]
by the formula
\[
i_{\tau,d}(m) = \{m\} \times d
\]

\begin{proposition}
For each $\tau \in \cT$ and $d \in D_\tau$, and let $i_{\tau,d}$ be the map defined above.  The following hold:
\begin{enumerate}
\item For each $\tau \in \cT$ and $d \in D_\tau$, the map $i_{\tau,d}$ is a homeomorphism of $\tau$ onto its image $i_{\tau,d}(\tau) \subset \WPSp$. 
\item For any face $\tau' \subset \tau$, the restriction of $i_{\tau,d}$ to $\tau'$ coincides with $i_{\tau',d'}$ for some $d' \in D_{\tau'}$.
\end{enumerate}
In other words, $\WPSp$ is a regular cell complex whose partially ordered set of cells is naturally isomorphic to $\widehat{\cT}$.
\end{proposition}

\begin{proof}
Note that the composite $\tau \to \WPSp \to \psprime$ is the usual inclusion of $\tau$ into $\psprime$---in particular $\tau \to i_{\tau,d}(\tau)$ is a continuous bijection.  Since $\tau$ is compact and $\WPSp$ is Hausdorff, this proves (1).  For (2), simply put $d' = \mathrm{res}_{\tau,\tau'}(d)$.
\end{proof}

\begin{remark}
In fact the Proposition shows that $\WPSp$ is a ``$\Delta$-complex'' in the sense of \cite[2.1]{Hatcher}, or a ``generalized simplicial complex'' in the sense of \cite[Definition 2.41]{Kozlov}.
\end{remark}

\begin{remark}
\label{rem:mapoutofcell}

We will use the following device for constructing continuous maps out of $\WPSp$ or $X_0$:
\begin{enumerate}
\item Let $K$ be a regular cell complex, let $\{\kappa\}$ be the poset of cells, and let $L$ be a topological space.  If $\{j_{\kappa}:\kappa \to L\}$ is a system of continuous maps such that $j_{\kappa}\vert_{\kappa'} = j_{\kappa'}$ whenever $\kappa' \subset \kappa$, then there is a unique continuous map $j:K \to L$ with $j\vert_{\kappa} = j_{\kappa}$ for all $\kappa$.  
\item Let $L$ be a topological space.  If $\{j_{\tau}:X_{0,\tau} \to L\}_{\tau \in \cT}$ is a system of continuous maps such that $j_{\tau} \vert_{\tau'} = j_{\tau'}$ whenever $\tau' \subset \tau$, then there is a unique continuous map $j:X_0 \to L$ with $j\vert_{\tau} = j_\tau$ for all $\tau$.
\end{enumerate}
In other words, $K$ is a colimit of its cells and $X_0$ is a colimit of the components $X_{0,\tau}$. \end{remark}

\begin{remark}
\label{rem:section}
For each $\tau \in \cT$, $i_{\tau,1}$ be the embedding $\tau \hookrightarrow \WPSp$ where the ``1'' in the subscript indicates the identity element of $D_\tau$.  These assemble to an inclusion $\psprime \hookrightarrow \WPSp$ by Remark \ref{rem:mapoutofcell}.
\end{remark}

\subsubsection{The homotopy type of $\WPSp$}

It is easy to identify the homotopy type of $\WPSp$, using the technique of ``shelling.''

\begin{theorem}
The regular cell complex $\WPSp$ has the homotopy type of a wedge of $n$-dimensional spheres.
\end{theorem}

\begin{proof}
We will show that $\WPSp$ is \emph{shellable} in the sense of \cite[Definition 12.1]{Kozlov}---then by \cite[Theorem 12.3]{Kozlov} $\WPSp$ is homotopy equivalent to a wedge of $n$-dimensional spheres.  By \cite[Proposition 1]{BM}, the triangulation $\cT$ of $\PSp$ has a shellable subdivision, denote it by $\cS$.  Let $\widehat{\cS}$ denote the lift of $\cS$ to $\WPSp$.  For each top-dimensional face $\sigma$ of $\cS$, fix a total order $F(\sigma,1),\ldots,F(\sigma,k)$.  Since $\WPSp \to \PSp$ is a branched covering along the simplices of $\cS$, whenever  $\sigma_1,\sigma_2,\ldots, \sigma_N$ is a shelling of $\cS$
 \[
F(\sigma_1,1),\ldots,F(\sigma_1,k_1),F(\sigma_2,1),F(\sigma_2,2),\ldots,F(\sigma_2,k_2),\ldots,F(\sigma_N,1),\ldots F(\sigma_N,k_N)
\]
is a shelling of $\widehat{\cS}$.
\end{proof}

\subsection{Embedding $\WPSp$ into $\X_0$}
\label{embeddingI}

In this section, using the positivity conditions on the coefficients $a_m$ of $f$
described below Equation (\ref{eq:f}), we will construct an embedding of $\WPSp$ into the special fiber $X_0$.

\subsubsection{General remarks on positive loci in toric varieties}
\label{subsubsec:generalremarks}
Let $T \cong (\bC^*)^n$ be an algebraic torus and fix a splitting $T \cong \mathrm{U}(1)^n \times \bR_{>0}^{n}$.  If $W$ is a toric variety acted on by $T$, and $1 \in W$ is a base point in the open orbit, then the \emph{positive locus} of $W$ is the $\bR_{>0}^n$-orbit of $1$ on $W$.  The \emph{nonnegative locus} is the closure of the positive locus in $W$.  We write $W_{>0}$ for the positive locus and $W_{\geq 0}$ for the nonnegative locus.  

\begin{example}
\label{ex:Wgeq0}
Let $W$ be an affine toric variety of the form $\Spec(\bC[M \cap \sigma])$.  Then under the identification $W \cong \Hom(M \cap \sigma,\bC)$ of Proposition \ref{prop:specpoints}, the nonnegative locus is
\begin{equation}
W_{\geq 0} \cong \Hom(M \cap \sigma,\bR_{\geq 0})
\end{equation}
\end{example}

When $W = \Proj(\bC[\bZ_{\geq 0}^{n+1}])$, the nonnegative locus is the set of points whose homogeneous coordinates can be chosen to be nonnegative real numbers.  It can be identified with a simplex.  The following Proposition investigates this example in more detail:

\begin{proposition}
\label{prop:moment}
Let $\tau \subset M_\bR$ be a lattice simplex, and let $\bP^{\dim(\tau)}$ be the projective space of Definition \ref{def:projYtoPtau}.  Let $[x_m]_{m \in \tau^{[0]}}$ be homogeneous coordinates for a point of $\bP^{\dim(\tau)}$.  Define the moment map $\mu_\tau: \bP^{\dim(\tau)}\rightarrow M_\bR$ by
\[
\mu_\tau([x_m]_{m \in \tau^{[0]}}) = 
\frac{\sum_{m \in \tau^{[0]}} |x_m|^2 m}{\sum_{m \in \tau^{[0]}} |x_m|^2}
\]
Then $\mu_\tau$ is a homeomorphism of $\bP^{\dim(\tau)}_{\geq 0}$ onto $\tau$.
\end{proposition}

\begin{proof}
See \cite[\S4.2]{Fu}
\end{proof}

\begin{remark}
The map of Proposition \ref{prop:moment} is the usual moment map for a Hamiltonian torus action and symplectic form on $\bP^{\dim(\tau)}$, but the conclusion of the Proposition holds for any map of the form
\[
\frac{\sum_{m \in \tau^{[0]}} |x_m|^e m}{\sum_{m \in \tau^{[0]}} |x_m|^e}
\]
so long as $e$ is real and $e> 0$.  When $e > 1$, these maps are smooth.  The case $e = 1$ may lead to a simpler formula for the map considered in Definition \ref{def:lambda}
\end{remark}

\begin{remark}
\label{rem:moment-compatible}
The moment maps of Proposition \ref{prop:moment} have the following compatibility feature: if $\tau' \subset \tau$ is a face of $\tau$, then the restriction of $\mu_{\tau}$ to $\bP^{\dim(\tau')}
\subset \bP^{\dim(\tau)}$ is $\mu_{\tau'}$.  In particular by Remark \ref{rem:mapoutofcell}, there is a well defined map 
\[
\nu:X_0 \to \PS' \subset M_\bR,
\]
such that, for all $\tau$, its restriction to $X_{0, \tau}$ is given by $\nu_{\tau}:=\mu_{\tau} \circ p_{\tau}$.
\end{remark}

\subsubsection{Embedding}

Recall the $D_\tau$-equivariant maps
\[
p_\tau:X_{0,\tau} \to \bP^{\dim(\tau)} \setminus V(\ell_\tau)
\]
of Proposition \ref{prop:degenaffine}.  We use it to define a \emph{nonnegative locus} in $X_{0,\tau}$.

\begin{definition}
\label{def:nonneglocus}
Fix $\tau \in \cT$.  Let $\bP^{\dim(\tau)}$ be the projective space of Definition \ref{def:projYtoPtau}, let $X_{0,\tau}$ be the affine variety of Definition \ref{def:ftau}.  We define subsets 
\[
\begin{array}{ccccc}
\bP^{\dim(\tau)}_{>0} & \subset & \bP^{\dim(\tau)}_{\geq 0} & \subset & \bP^{\dim(\tau)} \\
 (X_{0,\tau})_{>0}&  \subset & (X_{0,\tau})_{\geq 0}&  \subset &  X_{0,\tau}
 \end{array}
\]
as follows:
\begin{enumerate}
\item Let $\bP^{\dim(\tau)}_{>0} \subset \bP^{\dim(\tau)}$ be the set of points whose homogeneous coordinates can be chosen to be positive real numbers.  We call $\bP^{\dim(\tau)}_{>0}$ the \emph{positive locus} of $\bP^{\dim(\tau)}$.  
\item Let $\bP^{\dim(\tau)}_{\geq 0} \subset \bP^{\dim(\tau)}$ be the closure of $\bP^{\dim(\tau)}_{>0}$, i.e. the set of points whose homogeneous coordinates can be chosen to be nonnegative real numbers.  We call $\bP^{\dim(\tau)}_{\geq 0}$ the \emph{nonnegative locus} of $\bP^{\dim(\tau)}$.
\item If $(Y_{0,\tau})_{\geq 0}$ is as defined in Example \ref{ex:Wgeq0}, let $(X_{0,\tau})_{\geq 0} = X_{0,\tau} \cap (Y_{0,\tau})_{\geq 0}$, 
\end{enumerate}
\end{definition}

\begin{proposition}
\label{propembeddingIa}
Let $\bP^{\dim(\tau)}_{\geq 0}$ be as in Definition \ref{def:nonneglocus} and let $V(\ell_\tau)$ be as in Definition \ref{def:ftau}.  The following hold
\begin{enumerate}
\item $V(\ell_\tau)$ does not meet $\bP^{\dim(\tau)}_{\geq 0}$, i.e.
\[
V(\ell_\tau) \cap \bP^{\dim(\tau)}_{\geq 0} = \varnothing
\]
\item The projection of $X_{0,\tau}$ onto $\bP^{\dim(\tau)} \setminus V(\ell_\tau)$ induces a homeomorphism of nonnegative loci
\[
(X_{0,\tau})_{\geq 0} \stackrel{\sim}{\to} \bP^{\dim(\tau)}_{\geq 0}
\]
\end{enumerate}
\end{proposition}

\begin{proof}
Suppose $[x_m]_{m \in \tau^{[0]}}$ are homogeneous coordinates for a point $P \in \bP^{\dim(\tau)}$.  If $P$ belongs to the nonnegative locus, then by definition we may choose the $x_m$ to be real and nonnegative.  Moreover, at least one of the $x_m$ must be nonzero, say $x_{m_0}$.  Then evaluating $\ell_\tau$ on $P$ gives
\[
\ell_\tau(P) = \sum_{m \in \tau^{[0]}} a_m x_m \geq a_{m_0} x_{m_0} > 0
\]
since all the $a_m$ are positive real numbers.  In particular $\ell_\tau(P) \neq 0$.  This proves (1).

Let us prove (2).  Let $v_0,\ldots,v_{\dim(\tau)}$ be the vertices of $\tau$.  A point of $(Y_{0,\tau})_{\geq 0}$ is given by a monoid homomorphism $x:M \cap \bR_{\geq 0} \tau \to \bR_{\geq 0}$.  Since $\bR_{\geq 0}$ is divisible and $\tau^{[0]}$ is a basis for $\bR \tau$, $x$ is determined by its values on $\tau^{[0]}$, and the map
\[
x \mapsto (x(v_0),\ldots,x(v_{\dim(\tau)}))
\] 
is a homeomorphism of $(Y_{0,\tau})_{\geq 0}$ onto $\bR_{\geq 0}^{\dim(\tau)+1}$.  In these coordinates, the equation $f_{\tau} = 0$ defining $(X_{0,\tau})_{\geq 0}$ is
\[
\sum_{i = 0}^{\dim(\tau)} a_{v_i} x(v_i) = -a_0
\]
which (since $a_0 < 0$ and $a_{v_i} > 0$) is a simplex with a vertex on each coordinate ray of $\bR_{\geq 0}^{\dim(\tau)+1}$.  It follows that the projection onto $(\bR_{\geq 0}^{\dim(\tau)+1} \setminus \{0\})/\bR_{>0} \cong \bP^{\dim(\tau)}_{\geq 0}$ is a homeomorphism.
\end{proof}

To define an embedding $\WPSp \to X_0$, we may appeal to Remark \ref{rem:mapoutofcell} and define map it simplex by simplex.

\begin{definition}
\label{def:jtaud}
Let $\widehat{\cT}$ be the poset of Definition \ref{def:That}.  For each $(\tau,d) \in \widehat{\cT}$ define the map $j_{\tau,d}$ to be the composite
\[
\xymatrix{
\tau \ar[r]^{\mu^{-1}} & \bP^{\dim(\tau)}_{\geq 0} \ar[r]^{p_{\tau}^{-1}} & (X_{0,\tau})_{\geq 0} \ar[r]^{d} & (X_{0,\tau})_{\geq 0}
}
\]
where 
\begin{itemize}
\item $\mu^{-1}$ is the inverse homeomorphism to the map of Proposition \ref{prop:moment}
\item $p_\tau^{-1}$ is the inverse homeomorphism to the map of Proposition \ref{propembeddingIa}(2).
\item $d$ denotes the action of $d \in D_\tau$ on $X_{0,\tau}$ of Definition \ref{def:Dtau}.
\end{itemize}
\end{definition}

\begin{proposition}
\label{propembeddingIb}
Let $\WPSp$ be as in Remark \ref{rem:quotientofS}, let $\widehat{\cT}$ be as in Definition \ref{def:That}, and for each $(\tau,d) \in \widehat{\cT}$ let $i_{\tau,d}:\tau \hookrightarrow \WPSp$ be the inclusion defined in Section \ref{sec:regularcell} and let $j_{\tau,d}$ be the inclusion of \ref{def:jtaud}.  There is a unique map $j:\WPSp \hookrightarrow X_{0}$ such that  for all $(\tau,d) \in \widehat{\cT}$, the square
\[
\xymatrix{
\tau \ar@{^{(}->}[d]_{i_{\tau,d}} \ar[r]^{j_{\tau,d}} & X_{0,\tau} \ar[d] \\ 
\WPSp \ar[r]_{j} & X_0
}
\]

commutes.
\end{proposition}

\begin{proof}
By Remark \ref{rem:mapoutofcell}(1), it suffices to show that the maps $\tau \to X_0$ given by $j_{\tau,d}$ are compatible in the sense that $j_{\tau,d}\vert_{\tau'} = j_{\tau',\mathrm{res}_{\tau,\tau'}(d)}$ whenever $\tau' \subset \tau$. 
To see this, note that if $t'\in \tau'\subset \tau,$ then $\mu^{-1}$ carries
$t'$ to $\bP^{\dim\tau'}_{\geq 0}\subset \bP^{\dim\tau}_{\geq 0}$
(see Remark \ref{rem:moment-compatible}).
The proof of Proposition \ref{propembeddingIa} shows that $p_{\tau'}^{-1}$
and $p_{\tau}^{-1}$ agree on this locus.  Finally, the actions of $d$ and $\mathrm{res}_{\tau,\tau'}(d)$
are defined to agree on the result.

\end{proof}

\begin{remark}
\label{rem:deligne}
The inverse image above $\tau \subset \PSp$ of the map $\WPSp \to \PSp$ is a mild generalization (to Fermat hypersurfaces in weighted projective spaces) of the space considered in \cite[pp. 88--90]{Deligne}.
\end{remark}

\subsection{$\WPSp$ embeds in $X_0$ as a deformation retract.}
In this section we prove that the inclusion $\WPSp \hookrightarrow X_0$ is a deformation retract.  This is a ``degenerate'' case of our Main Theorem, and plays an important role in the proof.

\subsubsection{Lifting deformation retractions along branched covers}
\label{sec:branched}
Let us first discuss a path-lifting property of branched coverings:

\begin{definition}
\label{def:enter}
Let $W$ be a locally contractible, locally compact Hausdorff space and let $F_1 \subset F_2 \subset \cdots \subset F_k \subset W$ be a filtration by closed subsets.  
\begin{enumerate}
\item A map $p:W' \to W$ is \emph{branched along the filtration} $F$ if it is proper and if $p^{-1}(F_i \setminus F_{i-1}) \to F_i \setminus F_{i-1}$ is a covering space for every $i$.
\item A path $\gamma:[0,1] \to W$ is called an \emph{enter path for the filtration} $F$ if whenever $\gamma(t) \in F_i$, then $\gamma(s) \in F_i$ for all $s > t$.  (In other words once $\gamma$ enters the subset $F_i$, it does not leave).  Write $\Maps_F([0,1],W)$ for the space of enter paths for $F$, with the compact-open topology.

\item A deformation retraction $W \to \Maps([0,1],W)$ that factors through $\Maps_F([0,1],W)$ is called a \emph{$F$-deformation retraction}.
\end{enumerate}
\end{definition}

\begin{proposition}
\label{prop:branched-lifting}
Let $p:W' \to W$ be branched along a filtration $F$ of $W$.  Let $\gamma:[0,1] \to W$ be an enter path for $F$.  Then for each $w' \in p^{-1}(\gamma(0))$, there is a unique path $\tilde{\gamma}:[0,1] \to W'$ with $p \circ \tilde{\gamma} = \gamma$ and $\tilde{\gamma}(0) = w'$.  The path $\tilde{\gamma}$ is an enter path for $p^{-1}(F)$, and the map
\[
W' \times_{p,W,\mathrm{ev}_0} \Maps_F([0,1],W) \to \Maps_{p^{-1}(F)}([0,1],W')
\]
that sends $(w',\gamma)$ to the unique lift $\tilde{\gamma}$ is continuous.
\end{proposition}

\begin{proof}
This follows by a modification of the standard argument for covering spaces.  See \cite[Proposition 4.2]{Woolf} and also \cite{Fox}.
\end{proof}

\begin{corollary}
\label{cor:def-ret-lift}
Suppose $p:W' \to W$ is branched along a filtration $F$ of $W$.  Suppose that $r:W \to \Maps_F([0,1],W)$ is an $F$-deformation retraction (in the sense of Definition \ref{def:enter}) onto a subset $K \subset W$. Then $p^{-1}(K)$ is a $p^{-1}(F)$-deformation retract of $W'$.
\end{corollary}

\begin{proof}
The composite map
\[
W' \to W' \times_{p,W,\mathrm{ev}_0} \Maps_F([0,1],W) \to \Maps_F([0,1],W')
\]
where the first map is $w' \mapsto (w',r(p(w')))$ and the second map is the map of Proposition \ref{prop:branched-lifting} is a deformation retraction of $W'$ onto $p^{-1}(K)$.
\end{proof}

In proving Theorem \ref{thm:retract}, we will have to consider maps which have similar features to the branched covers of Section \ref{sec:branched}, except on each stratum they restrict to more general principal bundles. Lemma \ref{lem:slight-variant2} is a slight variant of Corollary \ref{cor:def-ret-lift}, which works for this larger class of maps as well.
\begin{lemma}[A slight variant of Corollary \ref{cor:def-ret-lift}]
\label{lem:slight-variant2}
Let $p:W_1 \to W_2$ be a continuous map, and let $K_2 \subset W_2$ be a closed deformation retract. Suppose that the restriction 
$
p^{-1}(W_2 \setminus K_2) \to W_2 \setminus K_2
$
is homeomorphic to the projection from a product $F\times (W_2 \setminus K_2)\ra W_2 \setminus K_2$.

Then $p^{-1}(K_2)$ is a deformation retract of $W_1$. 
\end{lemma}
\begin{proof}
Set $K_1 = p^{-1}(K_2)$.  Let us call a path $\gamma:[0,1] \to W_2$ a \emph{$K_2$-constant path} if it has the following property: if $\gamma(t) \in K_2$ then $\gamma(s) = \gamma(t)$ for all $s > t$.  In other words, once $\gamma$ enters $K_2$, it is constant.  Similarly let us define a $K_1$-constant path in $W_1$ if once it enters $K_1$, it is constant.

Using the product decomposition of $p^{-1}(W_2\setminus K_2)$, a $K_2$-constant path $\gamma:[0,1] \to W_2$ can be lifted in a canonical way to $\tilde{\gamma}:[0,1] \to W_1$ once the initial point $\tilde{\gamma}(0)$ is specified, and the assignment
\[
W_1 \times \{K_2\text{-constant paths in }W_2\} \to \{K_1\text{-constant paths in }W_1\}
\]
is continuous.

A strong deformation retraction of $W_2$ onto $K_2$ is given by a map $r:W_2 \to \Maps([0,1],W_2)$ such that
\begin{itemize}
\item $r(w)(0) = w$ for all $w$
\item $r(w)(1) \in K_2$ for all $w$
\item $r(w)(t) = w$ for all $w \in K_2$ and all $t$
\end{itemize}
For each $w$, the path $r(w):[0,1] \to W_2$ is a $K_2$-constant path.  Now we may define a map $r_1:W_1 \to \Maps([0,1],W_1)$ by the formula
\[
r_1(w_1) = \text{lift of $p \circ r_1(w_1)$ to $W_2$ with initial point $w_1$}.
\]
\end{proof}

\subsubsection{Retraction onto $\WPSp$}

\begin{definition}
\label{def:standardfiltration}
The \emph{standard toric filtration} of a toric variety $W$ is the filtration
\[
F_0 \subset F_1 \subset F_2 \subset \cdots  \subset W
\]
where each $F_i$ is the union of the torus orbits of dimension $i$ or less.
\end{definition}

\begin{proposition}
\label{prop:Pdimtau-ret}
Let $\tau \subset M$ be a lattice simplex, let $\bP^{\dim(\tau)}$ be the projective space of Definition \ref{def:projYtoPtau}, and let $\bP^{\dim(\tau)}_{\geq 0}$ be the nonnegative locus of $\bP^{\dim(\tau)}$ in the sense of Definition \ref{def:nonneglocus}.  Let $\ell$ be a homogeneous linear form on $\bP^{\dim(\tau)}$ that does not vanish on $\bP^{\dim(\tau)}_{\geq 0}$.  Let $F$ be the restriction of the standard toric filtration on $\bP^{\dim(\tau)}$ to $\bP^{\dim(\tau)} \setminus V(\ell)$.   Then 
\begin{enumerate}
\item There is an $F$-deformation retraction 
\[
r:\bP^{\dim(\tau)}\setminus V(\ell) \to \Maps_F([0,1],\bP^{\dim(\tau)}\setminus V(\ell))
\]
onto $\bP^{\dim(\tau)}_{\geq 0}$
\item $r$ may be chosen so that for any face $\tau' \subset \tau$, the restriction of $r$ to $\bP^{\dim(\tau')}$ is an $F'$-deformation retraction of $\bP^{\dim(\tau')} \setminus V(\ell')$ onto $\bP^{\dim(\tau')}_{\geq 0}$.  Here $\ell'$ is the restriction of $\ell$ to $\bP^{\dim(\tau')}$ and $F'$ is the restriction of $F$ to $\bP^{\dim(\tau')} \setminus V(\ell')$.
\end{enumerate}
\end{proposition}

\begin{proof}
For any two points $P,Q$ in $\bP^{\dim(\tau)} \setminus V(\ell)\cong \AA^{\dim(\tau)}$, let $\overline{PQ}$ be the real line segment between them.  Since each $F_i$ is an affine subspace, if $P$ and $Q$ are in $F_i$ then so is $\overline{PQ}$.  To produce an $F$-deformation retraction, it is enough to find a map $s:\bP^{\dim(\tau)} \setminus V(\ell) \to \bP^{\dim(\tau)}_{\geq 0}$ so that 
\begin{itemize}
\item $s(P) = P$ for all $P \in \bP^{\dim(\tau)}_{\geq 0}$.
\item $s(F_i) \subset F_i$ for all $i$
\end{itemize}
In that case the map $r$ given by $r(Q) = \overline{Q s(Q)}$ is an $F$-deformation retraction.  A suitable $s$ is given by the moment map of Proposition \ref{prop:moment}, and by Remark \ref{rem:moment-compatible}, the deformation retractions we build in this way will have property (2) of the Proposition.  
\end{proof}

\begin{theorem}
\label{thm:embedding}
The inclusion $\WPSp \hookrightarrow X_0$ admits a deformation retraction.
\end{theorem}

\begin{proof}
Since $X_{0,\tau} \to \bP^{\dim(\tau)} \setminus V(\ell_\tau)$ is branched along the standard toric filtration of $\bP^{\dim(\tau)}$, Proposition \ref{prop:Pdimtau-ret} and Corollary \ref{cor:def-ret-lift} together imply that $p_\tau^{-1}(\bP^{\dim(\tau)}_{\geq 0})$ is a deformation retract of $X_{0,\tau}$.  Moreover by part (2) of Proposition \ref{prop:Pdimtau-ret}, these deformation retractions are compatible with inclusions $X_{0,\tau'} \subset X_{0,\tau}$.  By Remark \ref{rem:mapoutofcell}, they therefore assemble to a deformation retraction of $X_0$ to $\WPSp$.
\end{proof}

\section{Log geometry and the Kato-Nakayama space}
\label{sec:log}

We recall the definition of a log space $X^\dag$ from \cite{Ka89} and the associated Kato-Nakayama space $X_\log$ from \cite{KN99} and \cite{NO10}.  We work with log structures in the analytic topology, which are treated in \cite{KN99}.

\subsection{Log structures and log smoothness}
\label{subsec:logdefs}
For us, a \emph{monoid} is a set with binary operation that is commutative, associative and has a unit. For each monoid $\sf M$, there is a unique group $\sf M^\gp$ called the \emph{Grothendieck group of $\sf M$} together with a map $\sf M\ra \sf M^\gp$ satisfying the universal property that every homomorphism from $\sf M$ to a group factors uniquely through $\sf M\ra \sf M^\gp$. A monoid is called \emph{integral} if $\sf M\ra \sf M^\gp$ is injective. Equivalently, the cancellation law holds in $\sf M$: $ab=ac\Rightarrow b=c$. 
A finitely generated and integral monoid is called \emph{fine}.
An integral monoid $\sf M$ is called \emph{saturated} if $x\in {\sf M}^\gp, x^n\in \sf M$ implies
$x\in \sf M$. A finitely generated, saturated monoid is called \emph{toric}. 
 
 \begin{example}
 If $\sigma \subset \bR^k$ is a rational polyhedral cone, then $\bZ^k \cap \sigma$ is a toric monoid. \end{example}
 
Let $X$ be an analytic space.
A \emph{pre-log structure} for $X$ is a sheaf of monoids $\shM_X$ together with a map of monoids $\alpha_X:\shM_X\ra\shO_X$ where we use the multiplicative structure for the structure sheaf. 
We call $(\shM_X,\alpha_X)$ a \emph{log structure} if $\alpha_X$ induces an isomorphism on invertible elements $\shM_X^\times \stackrel{\sim}{\ra} \shO_X^\times$. Given a (pre-)log structure $(\shM_X,\alpha_X)$, the triple $X^\dag = (X,\shM_X,\alpha_X)$ is called a (pre-)log space. Pre-log spaces naturally form a category on which we have a forgetful functor to the category of analytic spaces via $X^\dag\mapsto X$. This functor factors through the category of log spaces by the functor which associates a log structure to a pre-log structure. 
This is done by replacing $(\shM_X,\alpha_X)$ by the \emph{associated log structure} 
$(\shM^a_X, \alpha^a_X)$ given as
$$\shM^a_X=(\shM_X\oplus\shO^\times_X)/\{(m,\alpha_X(m)^{-1})\,|\,m\in\shM^\times_X=\alpha_X^{-1}\shO^\times_X\}$$
with $\alpha^a_X(m,f)=f\cdot \alpha_X(m)$.
Most of the time we will omit $\alpha_X$, assume it as known and refer to a log structure just by its sheaf of monoids.

\begin{example}
\label{ex:triviallog}
If $(X,\cO_X)$ is an analytic space, the \emph{trivial log structure} on $X$ is 
given by $\shM_X = (\cO_X)^\times,$ with $\alpha_X$ the inclusion map.
\end{example}

\begin{example}
\label{ex:divisoriallog}
If $(X,\cO_X)$ is an analytic space and $D\subset X$ a divisor, the \emph{divisorial log structure} $\shM_{(X,D)}$ on $X$ is 
given by $\shM_{(X,D)} = \cO_X\cap j_*\shO^\times_{X\setminus D},$ with $j:X\setminus D\ra X$ the open embedding and $\alpha_X$ the inclusion map.
\end{example}

\subsubsection{The standard toric log structure on a toric variety}
\label{subsubsec:logtoric}
Each toric variety $W$ has a natural divisor $D$ which is the complement of the open torus. Thus by Example~\ref{ex:divisoriallog}, $W$ carries the divisorial log structure $\shM_{(W,D)}$ which we call the standard log structure on $W$. We give another description for it in here.

\begin{definition}
\label{def:coherent}
A log space $(W,\shM_W)$ is called \emph{coherent} if each $x\in W$ has a neighborhood $U$ and a monoid $P$ with a map from the constant sheaf of monoids $P\ra \shM_U$ such that the pre-log structure associated to the composition $P\ra \shM_U\ra\shO_U$ coincides with the log structure $\shM_U$. The data $P\ra\shM_U$ is called a \emph{chart} of the log structure on $U$. 
\end{definition}

For a coherent log structure, we carry over properties of monoids.  For example, we call a coherent log structure \emph{fine} if there exists an open cover $\{U_i\}$ by charts $P_i\ra U_i$ with $P_i$ fine monoids.

If $\sigma \subset M_\bR$ is a rational polyhedral cone then the single chart
\[
M \cap \sigma \to \bC[M \cap \sigma]
\]
determines a coherent log structure on the affine toric variety $\Spec(\bC[M \cap \sigma])$.  If $W$ is any toric variety, these assemble to natural log structure on $W$ with charts induced by the canonical maps $\sigma \to \bC[M \cap \sigma]$ for each toric open set $\Spec \CC[M \cap \sigma]$ of $W$.  These log structures are fine and saturated.  For more see \cite[Example 2.6]{Ka96}.

\begin{example}
\label{ex:logaffineline}
The affine line $\bA^1 = \Spec(\bC[t])$ has a toric log structure whose chart $\bZ_{\geq 0} \to \bC[t]$ is given by $k \mapsto t^k$.  If $\cM_{\bA^1}$ denotes the sheaf of monoids and $U \subset \bA^1$ is an analytic open subset, then
\[
\Gamma(U,\cM_{\bA^1}) = \bigg\{
\begin{array}{ll}
\Gamma(U,\cO^\times) & \text{if $U$ does not contain $0$} \\
\bZ_{\geq 0} \oplus \Gamma(U,\cO^\times) & \text{if $U$ does contain $0$}
\end{array}
\]
\end{example}

\subsubsection{The log structure on a hypersurface}
\label{subsubsec:loghyper}

\begin{definition}
\label{def:logpullback}
If $u:X \to Y$ is a map of analytic spaces and $\shM_Y$ is a log structure on $Y$, the \emph{pullback log structure} is defined as the associated log structure to the pre-log structure given by the composition $u^{-1}\shM_Y\ra f^{-1}\shO_Y\ra\shO_X$.
\end{definition}

If $W$ is a toric variety and $Z \subset W$ is a hypersurface, we may pull back the log structure of Section \ref{subsubsec:logtoric} along the inclusion map $Z\hookrightarrow W$.  The charts of the log structure are of the form
\[
M \cap \sigma \to \bC[M \cap \sigma] \to \bC[M \cap \sigma]/f
\]
if $f = 0$ is the local equation of $Z$ in the chart $\Spec(\bC[M \cap \sigma]) \subset W$.

\begin{example}
\label{ex:logpoint}
If $(\bA^1,\cM_{\bA^1})$ is the log affine line of Example \ref{ex:logaffineline} and $0:\Spec \bC \to \bA^1$ is the origin, then the induced log structure on $\Spec \bC$ is given by the chart $\bZ_{\geq 0} \to \bC$ that carries each $k > 0$ to $0$.  This is the \emph{standard log point} of \cite[Definition 4.3]{Ka96}.  We denote it by $\Spec \bC^{\dagger}$.  The monoid is $\bZ_{\geq 0} \oplus \bC^*$. 
\end{example}

\subsubsection{Log smoothness}

A map of log spaces is called \emph{smooth} if it satisfies a lifting criterion for log first order thickenings.   A log space is smooth if the projection to a point with trivial log structure is smooth.  We do not recall the precise definitions here, see \cite[Section 3]{Ka96}.  A standard argument shows that many of the varieties and maps of Section \ref{sec:toricdegeneration} are log smooth.  We record the facts here.

Let $\Gamma_{\geq h}$ and $Y = \Spec(\bC[\tiM \cap \Gamma_{\geq h}])$ be as in Sections 
\ref{subsubsec:Gamma} and \ref{subsubsec:Y}, and let $\pi:Y \to \bA^1$ be the degeneration of Proposition \ref{prop:ambientdegeneration}.  Let $X \subset Y$ be the hypersurface of Section \ref{sec:32}.  We endow $Y$ with the log structure of Section \ref{subsubsec:logtoric} which we denote by $\cM_Y$, $\bA^1$ with the log structure of Example \ref{ex:logaffineline} which we denote by $\cM_{\bA^1}$, and $X$ with the log structure of Section \ref{subsubsec:loghyper} which we denote by $\cM_X$.

The map $\pi:Y \to \bA^1$ upgrades to a map of log spaces $\pi^\dagger:(Y,\cM_Y) \to (\bA^1,\cM_{\bA^1})$ in a unique way.  We abuse notation and also use $\pi^\dagger$ for the restriction $(X,\cM_X) \to (\bA^1,\cM_{\bA^1})$.

\begin{lemma}
\label{lem:smooth1}
The map $\pi^{\dagger}:(X,\cM_X) \to (\bA^1,\cM_{\bA^1})$ is log smooth.
\end{lemma}

Let $\overline{Y}$ and $\overline{X}$ be as in Section \ref{sec:tordegcompact} and let $\overline{X}$, and furnish them with the log structures of Section \ref{subsubsec:logtoric} and \ref{subsubsec:loghyper}.  We denote the log structure on $\overline{Y}$ by $\cM_{\overline{Y}}$ and the log structure on $\overline{X}$ by $\cM_{\overline{X}}$.  The maps $\overline{\pi}$ of Section \ref{sec:tordegcompact} upgrade to maps of log spaces $(\overline{Y},\cM_{\overline{Y}}) 
\to (\bA^1,\cM_{\bA^1})$ and $(\overline{X},\cM_{\overline{X}}) \to (\bA^1,\cM_{\bA^1})$.  We again abuse notation and denote both of these maps by $\overline{\pi}^{\dagger}$.

\begin{lemma}
\label{lem:smooth2}
The map $\overline{\pi}^{\dagger}:(\overline{X},\cM_{\overline X}) \to (\bA^1,\cM_{\bA^1})$ is log smooth.
\end{lemma}

Let $\cM_{X_0}$ denote the log structure on $X_0$ induced by $\cM_X$ under the inclusion map $X_0 \hookrightarrow X$.  Let $\cM_{\overline{X}_0}$ denote the log structure on $\overline{X}_0$ induced by $\cM_{\overline{X}}$ under the inclusion map $\overline{X}_0 \hookrightarrow \overline{X}$.  Then we have Cartesian diagrams of log spaces
\[
\begin{array}{cc}
{
\xymatrix{
(\overline{X}_0,\cM_{\overline{X}_0}) \ar[r] \ar[d] & (\overline{X},\cM_{\overline{X}}) \ar[d] \\
(\Spec \bC^{\dagger},\cM_{\Spec \bC^{\dagger}}) \ar[r] & (\bA^1,\cM_{\bA^1})
}
}
&
{
\xymatrix{
(X_0,\cM_{X_0}) \ar[r] \ar[d] & (X,\cM_{X}) \ar[d] \\
(\Spec \bC^{\dagger},\cM_{\Spec \bC^{\dagger}}) \ar[r] & (\bA^1,\cM_{\bA^1})
}
}
\end{array}
\]

We denote the maps $\overline{X}_0 \to \Spec \bC^{\dagger}$ and $X_0 \to \Spec \bC^{\dagger}$ by $\overline{\pi}_0^{\dagger}$ and $\pi_0^{\dagger}$ respectively.  We have a similar map $Y_0 \to \Spec \bC^{\dagger}$, and sometimes we abuse notation and denote it by $\pi_0^{\dagger}$ as well.

\begin{lemma}
\label{lem:smooth3}
The maps $\pi_0^{\dagger}$ and $\overline{\pi}_0^{\dagger}$ are log smooth.
\end{lemma}

\begin{remark}
We do not include the proofs of Lemmas \ref{lem:smooth1}, \ref{lem:smooth2}, and \ref{lem:smooth3} but note that they follow directly from K. Kato's toroidal characterization of log smoothness, see \cite[Theorem 4.1]{Ka96}.
\end{remark}

\subsection{The Kato-Nakayama space}

\begin{definition}
\label{def:Kato-Nakayama}
Let $W^\dag = (W,\cM_X,\alpha_X)$ be a log space.  Suppose $W^\dag$ is coherent in the sense of Definition \ref{def:coherent}.  The \emph{Kato-Nakayama space} is the space $W_{\log}$ whose underlying point set is
\[
W_\log = \{(x,h)\,\vert\,x\in W, h\in\Hom(\shM^\gp_{W,x},S^1), \text{ and }h(f)=\frac{f(x)}{|f(x)|}\hbox{ for any }f\in\shO^\times_{W,x}\}.
\]
topologized such that whenever $U \subset W$ is an open set and $P \to \cO_U$ is a chart, the embedding
\[
U_{\log} \hookrightarrow U \times \Hom(P^\gp,S^1) \qquad (x,h) \mapsto (x,h\vert_P)
\]
is a homeomorphism onto its image.  Let $\rho = \rho_{W^{\dagger}}$ denote the map $W_{\log} \to W$ given by $\rho(x,h) = x$.
\end{definition}

\begin{remark}
\label{rem:KNnoncoherent}
The above definition also makes sense when the log structure is not coherent, see \cite{NO10}. 
The point set definition is the same and the topology is the weak topology with respect to the functions $\rho$ and $(x, h) \ra h(m)$ for $m$ a
local section of $\shM_W$.
We will need this more general definition in Section~\ref{sec:five}.
\end{remark}

The map $\rho$ is continuous and surjective, and the construction $W^{\dagger} \mapsto W_{\log}$ is functorial such that for a morphism $W_1^{\dag} \to W_2^{\dag}$, the induced map $W_{1,\log} \to W_{2,\log}$ is continuous.

\begin{remark}
Define the \emph{K-N log point} $\Spec\CC_{KN}^\dag$ to be the analytic space $\Spec\CC$ with the log structure given by $\shM_{\Spec\CC}=\RR_{\ge 0}\times S^1$ and $\alpha_{\Spec\CC}:(r,h)\mapsto rh$.  Then there is a natural identification of sets
\[
W_{\log} = \Mor(\Spec\CC_{KN}^\dag,W^\dag)
\]
\end{remark}

\begin{example}
If $X$ carries the trivial log structure, then $X_\log = X$.
\end{example}

\begin{example}
\label{ex:exceptional}
Consider the affine line $(\AA^1,\cM_{\bA^1})$ of Example \ref{ex:logaffineline} and the standard log point $\Spec \bC^{\dag}$ of Example \ref{ex:logpoint}.  Then $(\AA^1)_{\log}$ is homeomorphic to $\bR_{\geq 0} \times S^1$, $(\Spec \bC^{\dag})_\log$ is homeomorphic to $S^1$, and the map $\rho:\bR_{\geq 0} \times S^1 \to \bC$ is given by $(r,e^{i\theta}) \mapsto re^{i\theta}$.  In other words the map $(\AA^1)_\log\ra \AA^1$ is the real oriented blowup of the origin and $(\Spec \bC^{\dagger})_{\log}$ is the ``exceptional circle'' of this blowup.  
\end{example}

In general when $W$ is a toric variety and is furnished with the log structure of Section \ref{subsubsec:logtoric}, the space $W_{\log}$ can be described in the manner of Proposition \ref{prop:specpoints}.

\begin{lemma} 
\label{lem:logtoric}
Let $P$ be a fine, saturated monoid and furnish $W=\Spec\CC[P]$ with the log structure given by the natural chart $P\ra \shO_W$.   Then $W_{\log}$ is naturally identified with $\Hom(P,\RR_{\geq 0} \times S^1)$.  Moreover, under the identification $W \cong \Hom(P,\bC)$ of Proposition \ref{prop:specpoints}, the map $\rho:W_{\log} \to W$ is given by composing with the monoid epimorphism $\RR_{\ge0}\times S^1\ra\CC$.
\end{lemma}

\begin{proof}
This is \cite{KN99},\,Ex.(1.2.11):
\end{proof}

From the description of Definition \ref{def:Kato-Nakayama} we obtain the following.
\begin{lemma} 
\label{lem:logpullback}
Let $(W_1,\shM_{W_1})$ be a log space, let $W_2 \to W_1$ be a morphism of complex analytic spaces and let $\cM_{W_2}$ be the pullback log structure on $W_2$ of Definition \ref{def:logpullback}.  The following diagram of topological spaces is Cartesian
$$
\xymatrix@C=30pt
{ W_{2,\log}\ar[r]\ar[d]&W_2\ar[d]\\
  W_{1,\log}\ar[r]&W_1\\
}
$$
\end{lemma}

\subsubsection{$Y_{\log}$ and $Y_{0,\log}$}
\label{subsubsec:Ylog}
In this Section and in Section \ref{subsubsec:Xlog}, we return to the degeneration of our hypersurface.  Here we describe $Y_{0,\log}$, the map $\rho:Y_{0,\log} \to Y_0$, and its fibers.

Let $\Gamma_{\geq h} \subset \tiM_{\bR}$ be the overgraph cone of Section \ref{subsubsec:Gamma}.  From Remark \ref{rem:specpoints}, we can describe $Y$ and $Y_0$ as spaces of monoid homomorphisms
\[
\begin{array}{rcl}
Y & = & \Hom(\tiM \cap \Gamma_{\geq h}, \bC) \\
Y_0 & = & \{\phi \in \Hom(\tiM \cap \Gamma_{\geq h},\bC) \mid \phi(0,1) = 0\}
\end{array}
\]
Then by Lemmas \ref{lem:logtoric} and \ref{lem:logpullback}
\[
\begin{array}{rcl}
Y_\log & = & \Hom(\tiM \cap \Gamma_{\geq h},\bR_{\geq 0} \times S^1) \\
Y_{0,\log} & = & \{\phi \in \Hom(\tiM \cap \Gamma_{\geq h}, \bR_{\geq 0} \times S^1) \mid \phi(0,1) \in \{0\} \times S^1\}
\end{array}
\]
\begin{enumerate}
\item $Y_{\log}$ is the space of monoid homomorphisms $\Hom(\tiM \cap \Gamma_{\geq h},\bR_{\geq 0} \times S^1)$
\item $Y_{0,\log}$ is the inverse image of $Y_0$ under the map $Y_{\log} \to Y$.  In other words, it is the space of monoid homomorphisms $\phi:\tiM \cap \Gamma_{\geq h} \to \bR_{\geq 0} \times S^1$ carrying $(0,1)$ to an element of $\{0\} \times S^1$.
\end{enumerate}

\subsubsection{$X_{\log}$ and $X_{0,\log}$}
\label{subsubsec:Xlog}

Lemmas \ref{lem:logtoric} and \ref{lem:logpullback} provide the following description of $ X_\log$ and $X_{0,\log}$.  Let $\Gamma_{\geq h} \subset \tiM_{\bR}$ be the overgraph cone of Section \ref{subsubsec:Gamma}.  Then 
\begin{enumerate}
\item
$X_\log$  is the inverse image of $X \subset Y$ under the map 
\[
\Hom(\tiM \cap \Gamma_{\geq h},\RR_{\ge 0}\times S^1) \to \Hom(\tiM \cap \Gamma_{\geq h},\bC)
\]
induced by $\bR_{\geq 0} \times S^1 \to \bC$.
\item $X_{0,\log}$ is the inverse image of $X_0 \subset Y$ under the map
\[
\Hom(\tiM \cap \Gamma_{\geq h},\bR_{\geq 0} \times S^1) \to \Hom(\tiM \cap \Gamma_{\geq h},\bC)
\]
induced by $\bR_{\geq 0} \times S^1 \to \bC$.
\end{enumerate}

The map $t^{\dagger}:(X,\cM_X) \to (\bA^1,\cM_{\bA^1})$ of Lemma \ref{lem:smooth2} induces a map $X_{\log} \to (\bA^1)_{\log}$.  For $c \in \bA^1_{\log} \cong \bR_{\geq 0} \times S^1$, let $(X_{\log})_c$ denote the fiber of this map above $c$.  

The map $\pi_0^{\dagger}$ of Lemma \ref{lem:smooth3} induces a map $X_{0,\log} \to \Spec \bC^{\dagger}$.  For $e^{i\theta} \in S^1 \cong \Spec \bC^{\dagger}$, denote the fiber of $\pi^{\dagger}_0$ by $(X_{0,\log})_\theta$.  Then
\begin{enumerate}
\item[(3)] $(X_{0,\log})_{\theta}$ is subset of $X_{0,\log}$ given by maps $\phi:\tiM \cap \Gamma_{\geq h} \to \bR_{\geq 0} \times S^1$ that carry $(0,1) \in \Gamma_{\geq h}$ to $(0,e^{i\theta}) \in \bR_{\geq 0} \times S^1$.
\end{enumerate}

The restriction of the map $\rho:X_{0,\log} \to X_0$ to any $(X_{0,\log})_{\theta}$ is surjective.

The map ${\overline\pi_Y}^\dagger$ (resp. $\overline{\pi}^\dagger$, $\overline{\pi}_0^\dagger$) induces a map of Kato-Nakayama spaces 
\[
(\overline\pi_Y)_\log:\overline{Y}_\log\ra\AA^1_\log\qquad 
(\hbox{resp. } \overline{\pi}_\log:\overline{ X }_\log\ra\AA^1_\log, \quad \overline{\pi}_{0,\log}:\overline{ X }_{0,\log}\ra 0_\log).
\]
Note that $\overline X_{0,\log}$ maps to $S^1=\Hom(\bZ_{\geq 0},S^1)=(\Spec\CC^\dagger)_\log$. 
We denote the restrictions of these maps to the loci inside $Y\subset \overline Y$
by removing the bar, e.g., $\pi_\log: X_\log\ra \AA^1_\log$.
For $c\in\AA^1_\log$ (resp. $0_\log$), we denote $(\overline X_\log)_c:=\overline\pi_{\log}^{-1}(c)$, $( X_{0,\log})_c:=\pi_{0,\log}^{-1}(c)$, etc. 
The fiber of $\pi_{0,\log}$ over $\theta\in S^1$ is given by 
$$( X_{0,\log})_\theta=
\{\phi\in\Hom(\tiM\cap \Gamma_{\ge0},\RR_{\ge 0}\times S^1)\,|\,\phi(0,1)=(0,\theta)\}\times_{Y} { X_0}$$
which surjects to $ X_0$ under $\rho$ for each $\theta$.

\subsection{Embedding the skeleton into the Kato-Nakayama space}

In this section we construct an embedding of the skeleton $S_{\sympol,\cT} \subset \PSp \times \Hom(M,S^1)$ (Definition \ref{def:skeleton}) into the fiber over $1$ of the Kato-Nakayama space of the degeneration $(X_{0,\log})_{1}$ (Section \ref{subsubsec:Xlog}).  We will first define a map
\[
\lambda:\PSp \times \Hom(M,S^1) \to Y_{\log}
\]
and then show that $\lambda$ restricts to an embedding $S_{\sympol,\cT}\hookrightarrow (X_{0,\log})_1$.  We use the description of $Y_{\log}$ given in Section \ref{subsubsec:Ylog}, i.e.
\begin{eqnarray*}
Y_{\log} & = & \Hom(\tiM \cap \Gamma_{\geq h},\bR_{\geq 0} \times S^1) \\
& = & \Hom(\tiM \cap \Gamma_{\geq h},\bR_{\geq 0}) \times \Hom(\tiM \cap \Gamma_{\geq h},S^1).
\end{eqnarray*}

\begin{definition}
\label{def:lambda}
Let $j:\WPSp \hookrightarrow X_{0}$ be the embedding of Proposition \ref{propembeddingIb}, and let us regard $\PSp$ as a subset of $\WPSp$ by the embedding of Remark \ref{rem:section}.  Define $\lambda:\PSp \times \Hom(M,S^1) \to Y_\log$ by the formula
\[
\lambda(x,\phi)(m,r) = (j(x)(m,r),\phi(m)) \in \bR_{\geq 0} \times S^1
\]
\end{definition}

\begin{remark}
\label{rem:jx}
In the Definition, we are regarding $j(x)$ as a homomorphism $\tiM \cap \Gamma_{\geq h} \to \bC$ in the manner of Proposition \ref{prop:specpoints}.  Proposition \ref{propembeddingIa} shows that $j_{\tau,1}$ maps $\tau$ homeomorphically onto $(X_{0,\tau})_{\geq 0}$ so in fact $j(x)$ is a homomorphism $\tiM \cap \Gamma_{\geq h} \to \bR_{\geq 0}$.  The monoid homomorphism $j(x)$ has a messy explicit formula.  It is given implicitly by the following rules
\begin{itemize}
\item $j(x)(m,r) = 0$ unless $r = h(m)$ and $m \in \bR_{\geq 0} \tau_x$
\item For $m \in \tau_x^{[0]}$, the values $j(x)(m,h(m))$ are the unique positive real solutions to the following system of equations:
\begin{eqnarray}
\sum_{m \in \tau^{[0]}} (j(x)(m,h(m)))^2 m & = & x \sum_{m \in \tau^{[0]}} (j(x)(m,h(m)))^2 \\
\sum_{m \in \tau^{[0]}} a_m j(x)(m,h(m)) & = & -a_0
\end{eqnarray}
The first equation comes from Proposition \ref{prop:moment} and the second ensures $j(x)\in X_{0,\tau_x}$.
\end{itemize}
\end{remark}

\subsubsection{Properties of the embedding $\lambda$}

\begin{proposition}
The image of $\PSp \times \Hom(M,S^1)$ under $\lambda$ is contained in $(Y_{0,\log})_1$.  The image of $S_{\sympol,\cT}$ under $\lambda$ is contained in $X_{0,\log}$.
\end{proposition}

\begin{proof}
By Definition \ref{def:lambda} and Remark \ref{rem:jx}, for any $x \in \PSp$ the homomorphism $j(x)$ carries $(0,1) \in \Gamma_{\geq h}$ to $0 \in \bR_{\geq 0}$.  From Definition \ref{def:lambda}, the homomorphism $\lambda(x)$ carries $(0,1)$ to $(0,1) \in \bR_{\geq 0} \times S^1$.  It follows that $\lambda(\PSp \times \Hom(M,S^1)) \subset (Y_{0,\log})_1$.  

The map $\lambda(x,\phi):\tiM \cap \Gamma_{\geq h} \to \bR_{\geq 0} \times S^1$ belongs to $(X_{0,\log})_1 \subset (Y_{0,\log})_1$ if and only if
\begin{equation}
\label{eq:summincT}
\sum_{m \in \cT^{[0]}} a_m j(x)(m,h(m))\cdot\phi(m) = -a_0
\end{equation}
Since (Remark \ref{rem:jx}) $j(x)(m,h(m)) = 0$ unless $m$ belongs to $\tau_x$, the left hand side of \eqref{eq:summincT} is
\[
\sum_{m \in \tau_x^{[0]}} a_m j(x)(m,h(m)) \phi(m)
\]
If $(x,\phi)$ belongs to $S_{\sympol,\cT}$, then $\phi(m) = 1$ for every $m \in \tau_x^{[0]}$, so that this is equal to $-a_0$ by Remark \ref{rem:jx}.

\end{proof}

\begin{proposition}
\label{prop:coset}
For each $d \in D_{\tau}$, let $dA_\tau \subset G_\tau$ denote the corresponding coset of $A_\tau$ (see Eq.~\ref{groupexseq}).  Let $r:\Hom(\tiM,S^1) \to \Hom(M,S^1)$ denote the restriction map induced by the inclusion $m \mapsto (m,0)$.  Then $r$ induces an isomorphism 
the set of homomorphisms $\psi:\tiM \to S^1$ obeying the conditions
\begin{enumerate}
\item $\psi(m,h(m)) = d(m)$ for $m \in \tau$
\item $\psi(0,1) = 1$
\end{enumerate}
to a coset of $A_\tau$.
\end{proposition}

\begin{proof}
Because of the second condition, $\psi$ is determined by its values
on $M\cong M\times\{0\} \subset \tiM.$
If $\psi$ and $\psi'$ obey both conditions, then $\psi/\psi' = 1$ on $\tau,$ which
characterizes $A_\tau.$
\end{proof}

\begin{theorem}
\label{thm:cartesian}
Let $j$ be the map of Proposition \ref{propembeddingIb}, let $\lambda$ be the map of Definition \ref{def:lambda}, and let $\rho_1$ be the map of Definition \ref{def:Kato-Nakayama}.  Then the square
\[
\xymatrix{
S_{\sympol,\triang}\ar[r]^\lambda \ar[d] & ( X_{0,\log})_1\ar[d]^{\rho_1} \\
\WPSp\ar[r]_j  &  X_0
}
\]
is Cartesian.  In particular, 
$\lambda\vert_{S_{\sympol,\cT}}:S_{\sympol,\cT} \to (X_{0,\log})_1$ is a closed embedding.  
\end{theorem}

\begin{remark} The above diagram can be used to define $S_{\sympol,\triang}$. Replacing $(X_{0,\log})_1$ by $(X_{0,\log})_\theta$, we obtain a skeleton $S_{\sympol,\triang,\theta}$ for any $\theta$.
In fact, one may replace $(X_{0,\log})_1$ by $X_{0,\log}$ in order to obtain the entire family of skeleta over $S^1$ by varying $\theta$. This gives the geometric realization of the monodromy operation of the family $X\ra\AA^1$ along a loop around $0\in\AA^1$.
\end{remark}

\begin{proof}[Proof of Theorem \ref{thm:cartesian}]
Fix $(x,d) \in \WPSp$.  Thus, $x \in \PSp$ and $d$ is a homomorphism $M \cap \bR_{\geq 0}{\tau_x} \to S^1$ carrying the vertices of $\tau_x$ to $1$.  If we regard $j(x)$ as a monoid homomorphism as in Remark \ref{rem:jx}, then $j(x,d)$ is the monoid homomorphism
\[
j(x,d)(m,k) = \bigg\{
\begin{array}{ll}
d(m) j(x)(m,k) & \text{if $k = h(m)$ and $m \in \bR_{\geq 0} \tau_x$} \\
0 & \text{otherwise}
\end{array}
\]
The fiber of the left vertical map above $(x,d)$ is a coset of $A_\tau$ in $G_\tau$.  We will show that $\lambda$ carries this homeomorphically onto the fiber of $\rho_1$ above $j(x,d)$.

Let $r:\tiM \cap \Gamma_{\geq h} \to \bR_{\geq 0}$ and $\psi:\tiM \cap \Gamma_{\geq h} \to S^1$ be the components of a point $(r,\psi) \in (X_{0,\log})_1 \subset Y_{\log}$.  

{\it Claim:}  The point $(r,\psi)$ belongs to $\rho_1^{-1}(j(x,d))$ if and only if
\begin{itemize}
\item $r(m,k) = 0$ unless $k = h(m)$ and $m \in \bR_{\geq 0} \tau_x$
\item $r(m,h(m)) = j(x)(m,h(m))$ when $m \in \bR_{\geq 0} \tau_x$
\item $\psi(m,h(m)) = d(m)$ when $m \in \tau_x$
and (because we have restricted $\rho$ to $(X_{0,\log})_1$) $\psi(0,1) = 1$.
\end{itemize}

Because of the claim, the fiber $\rho_1^{-1}(j(x,d))$ is naturally parameterized by the set of homomorphisms $\psi:\tiM \to S^1$ that obey the third condition on this list, which is a coset of $A_{\tau}$ in $G_\tau$ by Proposition \ref{prop:coset}.
This agrees with the preimage in $S_{\sympol,\cT}$ of $(x,d)$ under $\lambda.$

To prove the claim, note that by the definition of $j$, we have $\rho_1(r,\psi) = j(x,d)$ if and only if the following holds: for all $(m,k) \in \tiM \cap \Gamma_{\geq h}$,
\[
r(m,k)\psi(m,k) = \bigg\{
\begin{array}{ll}
d(m) j(x)(m,k) & \text{if $k= h(m)$ and $m \in \bR_{\geq 0} \tau$} \\
0 & \text{otherwise}
\end{array}
\]
In particular we must have $r(m,k) = 0$ unless $k = h(m)$ and $m \in \bR_{\geq 0} \tau_x$.  In this case
\[
r(m,h(m)) = \psi(m,h(m))^{-1}d(m) j(x)(m,h(m))
\]
Since $r(m,h(m))$ and $j(x)(m,k)$ are positive real numbers, we must have $\psi(m,h(m)) = d(m)$.

\end{proof}

If $\dim \tau_x = n$, then $\rho_1^{-1}(j(x,d))=A_{\tau_x}= 1$ by the above proof, so we see that we have the following corollary.

\begin{corollary}
\label{cor:localiso}
Let $(x,d)\in \WPSp$ and suppose $\dim(\tau_x)=n.$  Then, in
a neighborhood of $j(x,d),$ $X_0$
is smooth and  
$\rho_1$ is an isomorphism.
\end{corollary}

\subsection{$S_{\sympol,\triang}$ is a strong deformation retract}
\label{subsubsec:proofsimplicial}

In this section we prove that 
$S_{\sympol,\triang}$ embeds in $( X_{0,\log})_1$ as a strong deformation retract. 
Recall that Proposition \ref{propembeddingIb} and
Theorem \ref{thm:embedding}, together with
Remark \ref{rem:moment-compatible}, give the following diagram, 
\[
\xymatrix{
S_{\sympol,\triang}\ar[r]^{\lambda} \ar[d] & ( X_{0,\log})_1\ar[d]^{\rho_1}  \ar[dr]^{\nu \circ \rho_1 }\\
\WPSp \ar[r]^{j}  &  X_0 \ar[r]^{\nu} & \PSp.
}
\]

\begin{lemma}
\label{lem:thick}
For each simplex $\tau \subset \partial\sympol'$ of $\cT$, let $X_{0,\tau}$ be as in Definition \ref{def:ftau}, let $p_{\tau}:X_{0,\tau} \to \bP^{\dim(\tau)}$ be as in Proposition \ref{prop:degenaffine}, let $\mu:\bP^{\dim(\tau)}\ra\tau$ be the moment map, set $\nu=\mu\circ p_\tau$ and let $\bP^{\dim(\tau)}_{\geq 0}$ be the simplex of Section \ref{subsubsec:generalremarks}. 
Each of the following inclusions admit deformation retractions:
\begin{enumerate}
\item 
For each $\tau \in \triang$, the inclusion
\[
p_\tau^{-1}(\bP^{\dim(\tau)}_{\geq 0}) \cup \nu_{\tau}^{-1}(\partial \tau) \hookrightarrow X_{0,\tau}
\]

\item For each $k \leq n$, the inclusion
\[
\bigcup_{\tau \mid \dim(\tau) = k} p_{\tau}^{-1}(\bP_{\geq 0}^{\dim(\tau)}) \cup \nu_{\tau}^{-1}(\partial \tau) \hookrightarrow \bigcup_{\tau \mid \dim(\tau) = k} X_{0,\tau}
\]
\item 
For each $k \leq n$, the inclusion
\[
\WPSp \cup \bigcup_{\tau \mid \dim(\tau) = k} \nu_{\tau}^{-1}(\partial \tau) \hookrightarrow
\WPSp \cup \bigcup_{\tau \mid \dim(\tau) = k} X_{0,\tau}
\]
\end{enumerate}
\end{lemma} 

Before proving the Lemma, let us indicate what these spaces are in case $\sympol$ is the tetrahedron indicated in Figure \ref{fig:quartikdual}.

\begin{example}
Let $\sympol$ be the tetrahedron with vertices at $(1,0,0)$, $(0,1,0)$, $(0,0,1)$, and $(-1,-1,-1)$ with its unique lattice triangulation $\cT$.  For any of the 4 triangles $\tau \in \cT$ (the situation is symmetric), the map $p_{\tau}:X_{0,\tau} \to \bP^{\dim(\tau)}$ is an open embedding.  The spaces appearing in Lemma \ref{lem:thick}(1) can be described as follows:
\begin{itemize}
\item $X_{0,\tau}$ is, under $p_\tau$, isomorphic to the complement of a line $\ell \subset \bP^2$ that meets each of the coordinate lines transversely.
\item $p_{\tau}^{-1}(\bP_{\geq 0}^{\dim(\tau)})$ is a simplex in $X_{0,\tau}$
\item $\nu_{\tau}^{-1}(\partial \tau)$ is the union of the three coordinate lines, not including the three points that lie on $\ell$.
\end{itemize}
In fact $p_\tau^{-1}(\bP^{\dim(\tau)}_{\geq 0}) \cup \nu_{\tau}^{-1}(\partial \tau)$ is obtained from the cycle of three affine lines $\nu_{\tau}^{-1}(\partial \tau)$ by gluing a 2-simplex along a loop that generates the fundamental group of $\nu_{\tau}^{-1}(\partial \tau)$.  In particular this space, like $X_{0,\tau}$ that it is embedded in, is contractible.
\end{example}

\begin{proof}[Proof of Lemma \ref{lem:thick}]
Let us first show that $\bP^{\dim(\tau)}_{\geq 0} \cup \mu^{-1}(\partial \tau)$ embeds in $\bP^{\dim(\tau)} \setminus V(\ell_\tau)$ as a deformation retract. This can be seen as follows. We can write $\mu^{-1}(\partial \tau)$ as the union 
\[
\mu^{-1}(\partial \tau) = \bigcup_{\tau' \subsetneq \tau}\bP^{\dim(\tau')} \setminus V(\ell_{\tau'})
\]
By Proposition \ref{prop:Pdimtau-ret}, for any proper face $\tau'$ of $\tau$, the space $\bP^{\dim(\tau')} \setminus V(\ell_{\tau'})$ deformation retracts onto $\bP^{\dim(\tau')}_{\geq 0} \subset \bP^{\dim(\tau)}_{\geq 0}$ in a way that is compatible with the inclusions of smaller strata. This gives a deformation retraction 
\[
\mu^{-1}(\partial \tau) \rightarrow \partial (\bP^{\dim(\tau)}_{\geq 0}),
\]
that can be extended to a deformation retraction 
\[
\bP^{\dim(\tau)}_{\geq 0} \cup \mu^{-1}(\partial \tau) \rightarrow \bP^{\dim(\tau)}_{\geq 0}
\] 
by defining it to be the identity on $\bP^{\dim(\tau)}_{\geq 0}$. 
Since $\bP^{\dim(\tau)}_{\geq 0}$ is contractible, this implies that $\bP^{\dim(\tau)}_{\geq 0} \cup \mu^{-1}(\partial \tau)$ is contractible as well. 

The existence of a deformation retraction $\bP^{\dim(\tau)} \setminus V(\ell_\tau) \rightarrow \bP^{\dim(\tau)}_{\geq 0} \cup \mu^{-1}(\partial \tau)$ is then a consequence of standard facts about $CW$ complexes: any contractible subcomplex of a contractible $CW$ complex is a strong deformation retract, see e.g. \cite{Mc} Lemma 1.6.  Claim (1) can be proved by applying Lemma \ref{lem:slight-variant2} to $p_{\tau}$. In fact, by Proposition \ref{prop:degenaffine} (3), $p_{\tau}: X_{0, \tau} \rightarrow \bP^{\dim(\tau)} \setminus V(\ell_\tau)$ is unramified away from $\bP^{\dim(\tau)}_{\geq 0} \cup \mu^{-1}(\partial \tau)$.

We turn now to claim (2). Note that for any pair of distinct $k$-dimensional simplices $\tau_1, \tau_2,$ $X_{0, \tau_1} \cap X_{0, \tau_2} = \nu_{\tau_1}^{-1}(\partial \tau_1) \cap \nu_{\tau_2}^{-1}(\partial \tau_2)$. As a consquence, the retractions defined in (1) agree on the intersections of the various components: in fact, they restrict to the identity there. This guarantees that they assemble to give a retraction of $\bigcup_{\dim(\tau) = k}X_{0, \tau}$ onto $\bigcup_{\dim(\tau) = k}(p_\tau^{-1}(\bP^{\dim(\tau)}_{\geq 0}) \cup (\mu_\tau \circ p_\tau)^{-1}(\partial \tau))$, as desired. 
The last claim follows from the observation that 
\[
\WPSp \cap \bigcup_{\dim(\tau) = k}X_{0, \tau} \subset \bigcup_{\dim(\tau) = k} p_\tau^{-1}(\bP^{\dim(\tau)}_{\geq 0}) \cup \nu_\tau^{-1}(\partial \tau).
\] Thus, the retraction obtained in (2) can be extended to $\WPSp \cup \bigcup_{\dim(\tau) = k}X_{0, \tau}$, by setting it equal to the identity on $\WPSp$.
\end{proof}

\begin{theorem}
\label{thm:retract}
$S_{\sympol,\triang}$ embeds in $( X_{0,\log})_1$ as a strong deformation retract.
\end{theorem}
\begin{proof}
Let $\PS'_k$ be the $(n-k)$-skeleton of the stratification of $\PSp$ given by $\triang$, i.e. set  $\PSp_{k} := \sqcup_{\tau \in \triang \text{, }\dim\tau \leq n-k} \tau^{\circ}$. Note that $\nu^{-1}(\PSp_k) = \bigcup_{\dim(\tau) = n-k}X_{0, \tau}$. Applying Lemma \ref{lem:thick} (3), with $k$ equal to $n-1$, we obtain a retraction of $X_0$ onto $\WPSp \cup \nu^{-1}(\PSp_1)$. By Corollary \ref{cor:localiso}, $\rho_1$ is a homeomorphism over $X_0 - \nu^{-1}(\PSp_1)$, and thus in particular over $X_0 - (\WPSp \cup \nu^{-1}(\PSp_1))$. Lemma \ref{lem:slight-variant2}
then implies that 
\[
\rho_1^{-1}(\WPSp \cup \nu^{-1}(\PSp_1)) = S_{\sympol, \triang} \cup (\nu \circ \rho_1)^{-1}(\PSp_1)
\]
and this space embeds in $(X_{0,\log})_1$ as a deformation retract.

By Lemma \ref{lem:thick} (2), $\nu^{-1}(\PSp_1)$ retracts onto 
\[
\nu^{-1}(\PSp_2)\cup \bigcup_{\dim(\tau) = n-1} p_\tau^{-1}(\bP^{\dim(\tau)}_{\geq 0}).
\] Also, for all $\tau$, $\rho_1$ restricts to a projection from a product with fiber $A_{\tau}$ over $\nu^{-1}(\tau^{\circ})$ (see the proof of Theorem \ref{thm:cartesian}). Thus we can apply Lemma \ref{lem:slight-variant2} in the following way: using the notations of Lemma \ref{lem:slight-variant2}, set $W_1 = (\nu \circ \rho_1)^{-1}(\PSp_1)$, set $p = \rho_1\vert_{W_1}$, and 
\[
K_2 =  \nu^{-1}(\PSp_2) \cup \bigcup_{\dim(\tau) = n-1} p_\tau^{-1}(\bP^{\dim(\tau)}_{\geq 0}).
\] 
This gives a retraction of $(\nu \circ \rho_1)^{-1}(\PSp_1)$ onto 
\[
(\nu \circ \rho_1)^{-1}(\PSp_2) \cup \bigcup_{\dim(\tau)=n-1} (p_\tau \circ \rho_1)^{-1}(\bP^{\dim(\tau)}_{\geq 0}) .
\]
Note that $S_{\sympol, \triang} \cap (\nu \circ \rho_1)^{-1}(\PSp_1)$ is contained in the latter. 
This follows from the proof of Lemma \ref{lem:thick} (3), observing that $S_{\sympol, \triang} = \rho_1^{-1}(\WPSp)$, while 
\[
(\nu \circ \rho_1)^{-1}(\PSp_1) = \rho_1^{-1}\left(\bigcup_{\dim(\tau) = n-1}X_{0, \tau}\right).
\]
Thus, in the usual manner, we extend the retraction constructed in the previous paragraph to a retraction of $S_{\sympol, \triang} \cup (\nu \circ \rho_1)^{-1}(\PSp_1)$ onto $S_{\sympol, \triang} \cup (\nu \circ \rho_1)^{-1}(\PSp_2)$, by setting it equal to the identity on $S_{\sympol, \triang}$. Iterating this argument, by considering the preimage along $\nu \circ \rho_1$ of skeleta of the stratification of $\PSp$ of increasingly higher codimension, in $n$ steps we achieve a retraction of  $( X_{0,\log})_1$ onto $S_{\sympol, \triang}$.
\end{proof}

\subsection{Proof of Main Theorem (\ref{thm:1.1})}
\label{subsec:Kato-Nakayama}
We wish to relate the affine hypersurface $Z \cong X_1$ to the special fiber of the Kato-Nakayama space
$( X_{0,\log})_1$.  In fact these spaces are homeomorphic, as we now show by proving
that Kato-Nakayama space is a fiber bundle.  Together with Theorem \ref{thm:retract},
this establishes the Main Theorem (\ref{thm:1.1})
of the introduction.

We wish to show that the map $\pi_\log:X_{\log} \to \bR_{\geq 0} \times S^1$ is a topological fiber bundle.  Since $X_{\log}$ is not endowed with a smooth structure, and the fibers of $\pi_{\log}$ are not compact, this is not straightforward to check.  However the relative compactification $\overline\pi:\overline{X} \to \AA^1$ considered in Section \ref{sec:tordegcompact} admits a natural log structure, and the map $\overline\pi_\log:\overline{X}_{\log} \to \bR_{\geq 0} \times S^1$ is proper. We can check that $\overline\pi$ is a fiber bundle whose fibers are manifolds with boundary with the ``relative rounding theory'' of Nakayama and Ogus.  The relevant result for us is the following:

\begin{theorem}[Nakayama-Ogus]
\label{NOthm}
Let $W^\dag$ be a fine log space, let $(\bA^1)^\dag$ be the affine line with the log structure of Example \ref{ex:logaffineline}, and let $f:W^{\dag} \to (\bA^1)^{\dag}$ be a morphism of fine log spaces.  If $f$ is proper, separated, and smooth, then the map $f_{\log}:W_{\log} \to \bR_{\geq 0} \times S^1$ is a topological fiber bundle. 
\end{theorem} 

\begin{proof}
By \cite[Remark 2.2]{NO10}, any map $f$ satisfying the hypotheses of the Theorem is exact in the sense of loc. cit., Definition 2.1.  Then the Theorem is a special case of loc. cit. Theorem 5.1.
\end{proof}

Recall that an $n$-dimensional topological manifold with boundary is a topological space locally homeomorphic to either $\RR^n$ or $\RR^{n-1}\times \RR_{\ge 0}$.   If $W$ is a topological manifold with boundary write $W^{\circ}$ for the interior, i.e. the set of points with a neighborhood homeomorphic to $\bR^n$, and $\partial W$ for the complement of $W^{\circ}$.

\begin{proposition}
\label{prop:topmfdbdry}
\begin{enumerate} 
\item $\overline {X}_\log$ as well as $(\overline {X}_\log)_c$ for each $c \in \AA^1_\log$ is a topological manifold with boundary. 
\item For each $c \in \bR_{\geq 0 } \times S^1$, the interior of the fiber $(\overline{X}_{\log})_c$ is precisely $(X_{\log})_c$.
\end{enumerate}
\end{proposition}

\begin{proof} 
Recall that in \cite{NO10}, a morphism of monoids $\theta:P \to Q$ is called \emph{vertical} if the image of $P$ is not contained in any proper face of $Q$.  
The morphism is \emph{exact} if the diagram
$$
\xymatrix{
P\ar[r]\ar[d]&Q\ar[d]\\
P^\gp\ar[r]&Q^\gp
}
$$
is Cartesian.
A morphism $(W_1,\shM_1) \to (W_2,\shM_2)$ of log spaces is called \emph{vertical at $x \in W_1$} (resp. \emph{exact at $x \in W_1$}) if the induced map of monoids
$
\cM_{2,f(x)} \to \cM_{1,x}
$
is vertical (resp. exact). We have log smoothness of the maps in consideration by Lemma~\ref{lem:smooth1} and \ref{lem:smooth2}. 
Moreover, it is not hard to see that the maps are exact.  
Under these conditions, by \cite[Theorem 3.5]{NO10}, the fibers are manifolds with boundary and boundary points coincide with vertical points.  Thus, the Proposition is a consequence of the following claim:
\begin{quote}
A point of $\overline{X}$ is vertical for the map $\overline{\pi}^\dagger$ if and only if it belongs to $X \subset \overline{X}$.
\end{quote}
Indeed, recalling that $\ZZ_{\ge0}$ gives a chart on the base, we just need to check where the generator of $\ZZ_{\ge0}$ gets mapped into a proper face of a stalk of the log structures upstairs and this is precisely in $\overline{X}\setminus X$.

\end{proof}

\begin{corollary} The map $\overline \pi_\log:\overline X_\log\ra \AA^1_\log$ is a topological fiber bundle.
\end{corollary}
\begin{proof}
Both $\overline X ^\dagger$, $(\AA^1)^\dagger$ are fine log spaces.
The map $\overline\pi$ is proper, separated and exact by Theorem \ref{NOthm}.
Log smoothness is given by Lemma~\ref{lem:smooth3}.
\end{proof}

\begin{corollary}
\label{maincor}
The map $\pi_\log: X_\log\ra \AA^1_\log$ is a topological fiber bundle.
\end{corollary}

In particular, $( X_{0,\log})_1$ is homeomorphic to the hypersurface $Z = V(f)$.  By Theorem \ref{thm:retract},
$( X_{0,\log})_1$ deformation retracts to $S_{\sympol,\cT}$.  Therefore, so does $Z$.  We have thus
proven the Main Theorem (\ref{thm:1.1}) of the introduction.

\section{Hypersurfaces in affine toric varieties}
\label{sec:five}
We now consider a generalization of our setting and our theorem to address the case where
$Z = f^{-1}(0)$ is a smooth hypersurface in a general affine toric variety $A$.
Such an $A$ contains a dense algebraic torus $T$ and by 
Theorem~\ref{thm:retract} we already know a skeleton for $S$ of $Z\cap T$ upon fixing the origin and a triangulation of the Newton polytope of $f$.
It turns out that a skeleton for $Z$ can be given as a topological quotient space of $S$, so partial compactification translates into taking a quotient in terms of skeleta. 
This is what we are going to prove in this section.

\begin{example}
\label{ex:generalskeleton}
As a simple example of this more general setting, we can consider the polynomial $f: \bC^2 \rightarrow \bC$, $f(x,y)= x^2 + xy + y^2 - 1,$ that we discussed in Examples \ref{ex:degeneration}, \ref{ex:hypersurface}, and \ref{ex:c*c*skeleton}. We let $A = \bC^2$, and note that $A = \Spec(\bC[K\cap M])$ where
$K = \bR_{\geq 0}^2\subset \bR^2$ is a convex, maximal-dimensional cone. In Example \ref{ex:ccskeleton} we will work out the geometry of the zero locus of $f$ in $A$, and explictly describe a skeleton for it.  
\end{example}

\subsection{The general setup}

Let $M\cong \bZ^{n+1}$ be a lattice.
Let $M_\bR:=M\otimes_\bZ\bR \cong \bR^{n+1}.$
Let $K  \subset M_\bR$ be an $(n+1)$-dimensional, convex
rational polyhedral cone.
Then $M \cap K$ is a finitely generated monoid and $A := \Spec(\bC[M \cap K])$ is an affine toric variety.
The smallest torus orbit in $A$ is $\Spec(\bC[M \cap K^\times])$ where $K^\times$ denotes the maximal linear subspace contained in $K$. We set $a=\dim K^\times$ and $b=n+1-a$.
Consider the projection $M\ra M/(M\cap K^\times)$
and its real analog $p_{K^\times}:M_\RR \ra M_\RR/K^\times$. We set 
$\overline K=p_{K^\times}(K)$ and have
$$A\cong \Spec(\bC[M \cap K^\times])\times \Spec\CC[\overline K\cap (M/(M\cap K^\times))]$$
and the first factor is $a$-dimensional and the second $b$-dimensional.

\begin{remark}  
Note that $A$ is smooth if and only if $(K,M_\RR,M)$ is isomorphic to $(\RR^a\times \RR_{\ge 0}^b,\RR^{a+b},\ZZ^{a+b})$.
\end{remark}

Let $f\in \bC[M \cap K]$ be a regular function on $A$, and let $\sympol$ be the Newton polytope of $f$ and $Z :=f^{-1}(0) \subset A$.
We make the following assumptions:
\begin{assumptions}\label{assumptions}
\begin{enumerate}
\item $A$ is either smooth or has at most an isolated singularity. Note that the latter implies that $b=0$, if $A$ is singular,
\item $\dim\sympol=\dim K$, 
\item $\overline K=\RR_{\ge 0}p_{K^\times}(\sympol)$, so $\sympol$ generates the cone $K$ up to invertible elements,
\item $Z$ is smooth.
\end{enumerate}
\end{assumptions}

\begin{remark} 
\begin{enumerate}
\item We necessarily have $\sympol \subset K$. By assumption (4) above, $\sympol+K^\times$ contains $0$. We may thus assume without loss of generality that $0\in\sympol$ by multiplying $f$ with a suitable invertible element if neccessary (leaving $Z$ unchanged).
\item Note that if assumpton (2) above is violated then $Z$ splits as a product $Z_1\times Z_2$ where $\dim Z_1=\dim\sympol$, $Z_1$ has the same Newton polytope as $Z$ and $Z_2$ is isomorphic to $(\CC^*)^{a'}\times \CC^{b'}$ for suitable $a',b'$. Since $(S^1)^{a'}$ is a skeleton for $(\CC^*)^{a'}\times \CC^{b'}$, imposing assumption (2) loses no generality.
\item In the case where $A$ is smooth, note that assumption (3) above can always be achieved by a linear coordinate transformation of $A$.
\end{enumerate}
\end{remark}

\begin{example}
Let $M=\bZ^{n+1}.$
If $K = M_\bR$, then $A = (\bC^*)^{n+1}.$
If $a+b = n+1$ and $K = \bR^a\times \bR_{\geq 0}^b,$ then $A =
(\bC^*)^a\times \bC^b$. 
For an example of a singular ambient variety, take $n=1$ and put
$K  = \{x\geq |y|\}\subset \bR^2.$
Then $A = \bC^2/\bZ_2.$
\end{example}

As in equation \eqref{eq:f}, we assume that $\sympol$ is equipped with a lattice triangulation $\cT_\sympol$ and that
\begin{equation}
\label{eq:f2}
f = a_0 + \sum_{m \in \cT^{[0]}} a_m z^m.
\end{equation}
As for the previous sections, 
we assume $0\in \cT_{\sympol}^{[0]}$, $a_0\in\RR_{<0}$ and $a_m\in\RR_{>0}$ for $m\neq 0$ and that these coefficients are generic in the sense of Remark~\ref{remgeneric}.
We also assume to have a convex piecewise linear function $h:\bR_{\geq 0}\sympol\rightarrow \bR$
taking non-negative integral values on $M$ such that the maximal dimensional simplices in $\cT_\sympol$ coincide with the non-extendable closed domains of linearity of $h|_\sympol$.

\subsection{The general definition of the skeleton}
As before, let $\cT$ denote the subset of $\cT_\sympol$ of the cells not containing $0$.
Let $\PS'$ denote the union of the cells in $\cT$ and for $x\in\PS'$, let $\tau_x$ denote the smallest cell of $\cT$ containing $x$.  Recall $S_{\sympol,\cT}$ from Definition \ref{def:skeleton}
\[
S_{\sympol,\cT}=\{(x,\phi)\in\PS'\times\Hom(M,S^1)\,|\, \phi(v)=1\hbox{ whenever }v\hbox{ is a vertex of }\tau_x\}.
\]
For $x\in K$, we denote by $K_{x}$ the smallest face of $K$ containing $x$.

\begin{definition}
\label{def:skeleton5}
Let $S_{\sympol,\cT,K}$ denote the quotient of $S_{\sympol,\cT}$ by the equivalence relation $\sim$ given by
\[
(x,\phi)\sim(x',\phi')\iff x=x'\hbox{ and } \phi|_{K_{x}\cap M}= \phi'|_{K_{x}\cap M}
\]
\end{definition}

The goal is to show that $S_{\sympol,\cT,K}$ embeds in $Z$ as a deformation retract.

\begin{example}
 \label{ex:ccskeleton}
 Let us go back to the setup of Example \ref{ex:generalskeleton}.
 Recall that we have $K = \bR_{\geq 0}^2.$  Then $S_{\sympol,\cT,K}$ is a quotient of $S_{\sympol,\cT}$
 as in Definition \ref{def:skeleton5}.  The quotient is only nontrivial for $x=b$ and $x=e.$  For $x=b,$ $K_x$
 is the $x$-axis, and therefore $\beta\sim \beta'$ in $G_{\{b\}}$ (see Example \ref{ex:c*c*skeleton}), meaning
 the two circles are contracted to two points.  The same happens when $x=e$.  As a result, four of the
 five circles in the bouquet that is $S_{\sympol,\cT}$ are contracted, and $S_{\sympol,\cT,K}$ is homotopy
 equivalent to a single circle.
 
 As a reality check, we give an explict description of the geometry of this hypersurface, and verify that it does have the expected homotopy type.
 Solving   for $y$ in the equation $x^2 + xy + y^2 = 1$, presents the solution space as a branched cover of the $x$-plane with two branch points $x = \pm \frac{2}{\sqrt{3}}$.     
That space retracts to a $2:1$ cover of the line segment between the points, branched at the ends:  a circle. From this analysis it becomes clear that the restriction of the hypersurface to the algebraic torus $(\bC^*)^2$ removes
 the four points $(0,\pm 1)$ and $(\pm 1,0)$, which up to homotopy adds four circles to the hypersurface. This confirms that the calculation of the skeleton $S_{\sympol,\cT}$ contained in Example \ref{ex:c*c*skeleton} is correct.
\end{example}

\subsection{Construction of the ambient degeneration} 
The construction of the degeneration in the general case is not different from the previous. For completeness, we repeat it here.
Recall the notation $\tiM=M\oplus\ZZ, \tiM_\RR=\tiM\otimes_\ZZ\RR$.
As in Section~\ref{sec:tordegcompact}, we define the non-compact polyhedron
$$\overline\Gamma=\{(m,r)\,|\,m\in\sympol,r\ge h(m)\}\subset \tiM_\RR$$
and $\overline Y$ be the toric variety given by the normal fan of $\overline\Gamma$.
We may set $\Gamma_{\ge h}=\RR_{\ge0}\overline\Gamma$ and find the affine chart
$$Y=\Spec\CC[\Gamma_{\ge h}\cap \tiM]$$
on which we have the two regular functions $t=z^{(0,1)}$ and $\tilde f=\sum_{m\in\sympol}a_mz^{(m,h(m))}$.
In fact, $t$ extends to a regular function on $\overline Y$.
Let $X=V(\tilde f)$ denote the affine hypersurface cut out by $\tilde f$ in $Y$
and let $\overline X$ denote its closure in $\overline Y$. We restrict $t$ to a regular function on $\overline X$.

The following lemma elucidates the relation between $\overline\Gamma$ and $K$.
\begin{lemma} 
\label{lem:KGamma}
We have an inclusion preserving bijection
$$\{\hbox{faces of }\overline\Gamma\hbox{ containing }\overline\Gamma\cap(K^\times\times\RR)\}
\longleftrightarrow
\{\hbox{faces of }K\}$$
by sending a face $G$ on the left hand side to $(\RR_{\ge0}G+(K^\times\times\RR))\cap M_\RR$ on the right.
\end{lemma}
\begin{proof} Faces of $K$ are in inclusion-preserving bijection with faces of $K\times\RR$ and the latter coincides with the localization $\overline\Gamma+(K^\times\times\RR)$ of $\overline\Gamma$ by Assumptions~\ref{assumptions}(3).
\end{proof}

\subsection{The non-standard log structure}
Let $D$ denote the complement of the open torus in $\overline Y$. Then $D$ is a toric divisor in $\overline Y$.
In Section~4, we used the standard toric log structure $\shM_{\overline Y}=\shM_{(\overline Y,D)}$ on the toric variety $\overline Y$ (Section~\ref{subsubsec:logtoric}), which eventually led to an embedding
of $S_{\sympol,\cT}\subset (\overline Y_{0,\log})_1$ as a deformation retract. To indicate that $\overline Y_\log$ is defined using the log structure $\shM_{\overline Y}$, we denote it from now on by
$\overline Y(\shM_{\overline Y})_\log$. 

We now construct another log structure $\shF_{\overline Y}$ on $\overline Y$. 
For this we specify a reduced toric divisor $D_\shF\subset \overline Y$, i.e. $D_\shF\subseteq D$, and we then define $\shF_{\overline Y}$ as the divisorial log structure with respect to $D_\shF$.
Recall that the components of $D$ correspond to the facets of $\overline\Gamma$. To define $D_\shF$, we need to pick a subset of these facets.
\begin{definition} 
We let $D_\shF\subset\overline Y$ be the reduced toric divisor whose components correspond to the facets of $\overline\Gamma$ that do not contain the face $(K^\times\times\RR) \cap\overline\Gamma$ and define $\shF_{\overline Y}=\shM_{(\overline Y,D_\shF)}$.
\end{definition}

We want to describe the stalks of $\shF_{\overline Y}$ explicitly. Let $F_1,...,F_r$ be an enumeration of the facets of $\overline\Gamma$ containing $(K^\times\times\RR) \cap\overline\Gamma$. 
For a face $G\subseteq\overline\Gamma$, we denote by $\langle G,F_i\rangle$, the smallest face of $\overline\Gamma$ containing $G$ and $F_i$, i.e.
$$\langle G,F_i\rangle=\left\{\begin{array}{ll} F_i&\hbox{if }G\subseteq F_i,\\ \overline\Gamma&\hbox{otherwise.}\end{array}\right.$$
We define $F_G:=\bigcap_{i=1}^r\langle G,F_i\rangle.$
For faces $G_1,G_2\subseteq\overline\Gamma$ with $G_1\subseteq G_2$, we have $F_{G_1}\subseteq F_{G_2}$.
\begin{lemma} 
\label{lem:FGequiv}
We have
$F_G=\bigcap_{G\subseteq F_i}F_i=\langle\overline\Gamma\cap (K^\times\times\RR),G\rangle$
\end{lemma}

We can now identify the stalks of $\shF_{\overline Y}$.
\begin{lemma}
\label{lem:stalksF}
Let $y\in\overline Y$ be a point and $G\subseteq\overline\Gamma$ be the face that corresponds to the torus orbit that contains $y$.
We have 
$$\shF_{\overline Y,y}=(\tiM\cap\RR_{\ge0}(F_G-G))\otimes_{(\tiM\cap\RR_{\ge0}(F_G-G))^\times}\shO^\times_{\overline Y,y}$$ 
and this is a face of 
$$\shM_{\overline Y,y}=(\tiM\cap\RR_{\ge0}(\overline\Gamma-G))\otimes_{(\tiM\cap\RR_{\ge0}(\overline\Gamma-G))^\times}\shO^\times_{\overline Y,y}.$$
\end{lemma}
\begin{proof} On the chart $\tiM\cap\RR_{\ge0}(\overline\Gamma-G)\ra\CC[\tiM\cap\RR_{\ge0}(\overline\Gamma-G)]$ of the log structure $\shM_{\overline Y}$, 
the subsheaf $\shF_{\overline Y}$ up to invertible elements is generated by the those monomials that do not vanish on the divisors corresponding to $F_1,...,F_r$, i.e. precisely the monomials contained in $\RR_{\ge 0}(F_G-G)$. Moreover, $\RR_{\ge 0}(F_G-G)$ is clearly a face of $\RR_{\ge 0}(\overline\Gamma-G)$.
\end{proof}

The log structure $\shF_{\overline Y}$ will in general not be coherent. However we have the following replacement:

\begin{proposition} 
\label{lem:relcoherent}
The log structure $\shF_{\overline Y}$ is \emph{relatively coherent} in $\shM_{\overline Y}$ in the sense of \cite[Def.\,3.6,\,1.]{NO10}.
\end{proposition}
\begin{proof} 
This just states that $\shF_{\overline Y}$ is a sheaf of faces in $\shM_{\overline Y}$ which is Lemma~\ref{lem:stalksF}.
\end{proof}

Let $\shF_{\overline X}$ (resp. $\shM_{\overline X}$) denote the pullback of the log structure $\shF_{\overline Y}$ (resp. $\shM_{\overline X}$) to $\overline X$.
\begin{corollary} The log structure $\shF_{\overline X}$ is \emph{relatively coherent} in $\shM_{\overline X}$.
\end{corollary}

\subsection{Relative log smoothness}
Note that $t=z^{(0,1)}$ is a global section of $\shF_{\overline Y}$ since all $F_i$ contain $(0,1)$. Thus, by mapping the generator of $\ZZ_{\ge 0}$ to $t$, we obtain a map of log spaces
$\overline \pi^\dag:(\overline Y,\shF_{\overline Y})\ra (\AA^1,\shM_{\AA^1})$. Moreover the inclusion $\shF_{\overline Y}\subseteq\shM_{\overline Y}$ induces a map $g^{\overline Y}$ so that we have the sequence of maps of log spaces
$$(\overline Y,\shM_{\overline Y})\stackrel{g^{\overline Y}}{\lra} (\overline Y,\shF_{\overline Y})\stackrel{\overline \pi^\dag}{\lra} (\AA^1,\shM_{\AA^1})$$
and we know that the composition $\overline \pi^\dag\circ g^{\overline Y}$ is log smooth by an analogue of Lemma~\ref{lem:smooth2} for $\overline Y$. Recall from \cite[Def.\,3.6,\,2.]{NO10} the definition of a relatively log smooth map.

\begin{lemma} 
\label{lem:logsmY}
If $A$ is smooth, then the map $\overline \pi^\dag$ is relatively log smooth. If $A$ is not smooth, then $\overline \pi^\dag$ is relatively log smooth away from the closure of the torus orbit in $\overline Y$ corresponding to $(0\times\RR)\cap\overline \Gamma$.
\end{lemma}
\begin{proof} It remains to show that the stalks of $\shM_{\overline Y}/\shF_{\overline Y}$ are free monoids at points for which we claim the map to be relatively log smooth. 
Let $y\in \overline Y$ be a point in a torus orbit corresponding to a face $G\subseteq \overline \Gamma$. 
By Lemma~\ref{lem:stalksF}, we have
$\shM_{\overline Y,y}/\shF_{\overline Y,y}=(\tiM\cap\RR_{\ge0}(\overline\Gamma-G))/(\tiM\cap\RR_{\ge0}(F_G-G))$ and we need to show that this is isomorphic to $\ZZ_{\ge 0}^s$ for some $s$.
This is equivalent to saying that $\overline Y$ is smooth in a neighborhood of the torus orbit corresponding to to the smallest face of $\overline\Gamma$ that contains $F_G$ and $G$.
It suffices to show that for any subset $I\subseteq\{1,...,r\}$, $Y$ is smooth in a neighborhood of the torus orbit corresponding to $F_I:=\Gamma_{\ge h}\cap\bigcap_{i\in I}F_i$ except for the case where $F_I=\{0\}\times\RR_{\ge 0}$ because we make no claim for this by the restrictions made in the assertion in the lemma. 
Note that since $F_I$ contains $(K^\times\times\RR)\cap\Gamma_{\ge h}$, the torus orbit corresponding to $F_I$ is contained in the open subset $A\times\CC^*$ of $\overline Y$, so the statement follows from the smoothness of $A$ in codimension one.
\end{proof}
Note that we also have a sequence of log spaces
$$(\overline X,\shM_{\overline X})\stackrel{g^{\overline X}}{\lra} (\overline X,\shF_{\overline X})\stackrel{\overline \pi^\dag}{\lra} (\AA^1,\shM_{\AA^1})$$
where we abuse notation by denoting the second map as $\overline \pi^\dag$ again.
Again, we know that the composition $g^{\overline X}\circ\overline \pi^\dag$ is log smooth by Lemma~\ref{lem:smooth2}.
When $A$ is singular, note that $\overline X$ is disjoint from the torus orbit in $\overline Y$ corresponding to $(0 \times \RR) \cap \overline\Gamma$, so using Assumptions~\ref{assumptions}\,(4), we conclude the following.
\begin{lemma} 
\label{lem:logsmX}
The map $\overline \pi^\dag:(\overline X,\shF_{\overline X})\ra(\AA^1,\shM_{\AA^1})$ is relatively log smooth.
\end{lemma}

\subsection{The Kato Nakayama space is a fiber bundle}
By Remark~\ref{rem:KNnoncoherent}, we may construct the Kato-Nakayama space for $(\overline Y,\shF_{\overline Y})$ and 
$(\overline X,\shF_{\overline X})$ and by functoriality, we have maps
$$\overline Y(\shM_{\overline Y})_\log  \stackrel{g^{\overline Y}_\log}{\lra} \overline Y(\shF_{\overline Y})_\log \stackrel{\overline \pi^\dag_\log}{\lra} \AA^1_\log,$$
$$\overline X(\shM_{\overline X})_\log  \stackrel{g^{\overline X}_\log}{\lra} \overline X(\shF_{\overline X})_\log \stackrel{\overline \pi^\dag_\log}{\lra} \AA^1_\log.$$
The statement of Theorem.~\ref{NOthm} in \cite{NO10} allows for the weaker assumption of relative coherency of the source and relatively smoothness of the map, 
so we conclude from Lemma~\ref{lem:logsmX} along the same lines as in Section~\ref{subsec:Kato-Nakayama} the following result.

\begin{theorem} 
\label{thm:fiberbdlgen}
The maps of Kato-Nakayama spaces
$$\overline X(\shF_{\overline X})_\log \stackrel{\overline \pi^\dag_\log}{\lra} \AA^1_\log,$$
$$\overline X(\shM_{\overline X})_\log \stackrel{\overline \pi^\dag_\log\circ g^{\overline X}_\log}{\lra} \AA^1_\log$$
are topological fiber bundles.
\end{theorem}
Moreover, the statement of Proposition~\ref{prop:topmfdbdry} holds word for word when replacing $\overline X_\log$ and $X_\log$ respectively by $\overline X(\shF_{\overline X})_\log$ and $X(\shF_{X})_\log$ where
$\shF_{X}$ is the restriction of $\shF_{\overline X}$ to $X$.

\subsection{Embedding the skeleton $S_{\sympol,\cT,K}$ in the Kato-Nakayama space $(\overline X(\shF_{\overline X})_\log)_1$}
We use the following notation for the induced maps on Kato-Nakayama spaces
\[
\xymatrix{
\overline Y(\shM_{\overline Y})_\log \ar^{g^{\overline Y}_\log}[r]\ar_{\rho(\shM_{\overline Y})}[dr] & \overline Y(\shF_{\overline Y})_\log \ar^{\rho(\shF_{\overline Y})}[d] \\
& \overline Y
}
\]
and similarly for $X,\overline X$ in place of $Y,\overline Y$.
\begin{proposition} 
\label{cor_stalk_of_shF}
Given a point $y\in\overline Y$ contained in the torus orbit associated to the face $G\subseteq \overline\Gamma$, the map 
$g_\log|_{\rho(\shM_{\overline Y})^{-1}(y)}:\, \rho(\shM_{\overline Y})^{-1}(y)\ra \rho(\shF_{\overline Y})^{-1}(y)$ is the restriction map
$$ 
\Hom\left(\frac{\tiM\cap\RR_{\ge0}(\overline\Gamma-G)}{\tiM\cap\RR(G-G)},S^1\right)
\ra
\Hom\left(\frac{\tiM\cap\RR_{\ge0}(F_G-G)}{\tiM\cap\RR(G-G)},S^1\right)
$$
induced by the injection 
$\frac{\tiM\cap\RR_{\ge0}(F_G-G)}{\tiM\cap\RR(G-G)}\hra\frac{\tiM\cap\RR_{\ge0}(\overline\Gamma-G)}{\tiM\cap\RR(G-G)}$.
\end{proposition}
\begin{proof} This is a straight-forward combination of Def.~\ref{def:Kato-Nakayama} and Lemma~\ref{lem:stalksF} with the additional observation that 
$(\RR_{\ge0}(\overline\Gamma-G))^\times=(\RR_{\ge0}(F_G-G))^\times=\RR(G-G)$.
\end{proof}

As before, we denote by $(\overline X(\shF_{\overline X})_\log)_1$ the fiber of $g^{\overline X}_\log$ over $(0,1)\in\RR_{\ge 0}\times S^1=\AA^1_\log$.
There is a surjection
$$g^{\overline X}_\log:\overline X(\shM_{\overline X})_{\log,1}\ra\overline X(\shF_{\overline X})_{\log,1}.$$
\begin{theorem}
\label{thm:embgen}
We have a canonical embedding of $S_{\sympol,\cT}$ in $\overline X(\shM_{\overline X})_{\log,1}$ whose image under $g^{\overline X}_\log$ is canonically identified with $S_{\sympol,\cT,K}$.
\end{theorem}
\begin{proof} We have the result of Theorem~\ref{thm:cartesian} already, so in particular an embedding of $\widehat{\partial\sympol'}$ in $\overline X$ and of $S_{\sympol,\cT}$ in 
$(\overline X(\shM_{\overline X})_\log)_1$.
We need to show that the image of $S_{\sympol,\cT}$ under $g^{\overline X}_\log$ yields the quotient space $S_{\sympol,\cT,K}$.
We fix a point $x\in\overline X$ in a torus orbit $O_G=\Spec\CC[\RR(G-G)\cap\tiM]$ with $G\subseteq\overline\Gamma$.
Let us regard the composition
$$
\overline X(\shM_{\overline X})_{\log,1} 
\stackrel{g^{\overline X}_\log}{\lra}  
\overline X(\shF_{\overline X})_{\log,1}
\stackrel{\rho(\shF_{\overline X})}{\lra}
\overline X.
$$
By the Cartesian property of the Kato-Nakayama space in Lemma~\ref{lem:logpullback}, we may use Proposition~\ref{cor_stalk_of_shF} to identify
the restriction of $g^{\overline X}_\log$ to the inverse images of $x$ as the map $T_1\ra T_2$ where
$$T_1=\left\{\alpha\in \Hom\left(\frac{\tiM\cap\RR_{\ge0}(\overline\Gamma-G)}{\tiM\cap\RR(G-G)},S^1\right) \,\Bigg|\, \alpha(0,1)=1\right\},$$
$$T_2=\left\{\alpha\in \Hom\left(\frac{\tiM\cap\RR_{\ge0}(F_G-G)}{\tiM\cap\RR(G-G)},S^1\right) \,\Bigg|\, \alpha(0,1)=1\right\}.$$
Let $p:\tiM_\RR\ra M_\RR$ denote the natural projection and $K_{G}$ denote the smallest face of $K$ containing $p(G)$.
We have $p(\overline\Gamma)=\sympol$. We use the fact that the condition $\alpha(0,1)=1$ in $T_1,T_2$ can be replaced by changing the source of $\alpha$ to a subquotient of $M$ instead of $\tiM$. Precisely,
$$T_1=\Hom\left(\frac{M\cap\RR_{\ge0}(\sympol-p(G))}{M\cap p(\RR(G-G))},S^1\right),$$
$$T_2=\Hom\left(\frac{M\cap\RR_{\ge0}(p(F_G)-p(G))}{M\cap p(\RR(G-G))},S^1\right).$$

Note that if $x\in\widehat{\partial\sympol'}$ then $K_{G}$ coincides with $K_{x}$. 
Moreover, $F_G$ contains $(K^\times\times\RR)\cap\overline\Gamma$ and thus corresponds to the face $(\RR_{\ge0}F_G+K^\times\times\RR)\cap M_\RR$ of $K$ by Lemma~\ref{lem:KGamma}. 
We claim that this face is $K_G$. Indeed by Lemma~\ref{lem:FGequiv}, $F_G$ is the smallest face of $\overline\Gamma$ containing $G$ and $\overline\Gamma\cap K^\times\times\RR$ which maps to $K_{p(G)}$ under the bijection in Lemma~\ref{lem:KGamma}. 
Finally, we may assume that $G$ contains $\overline\Gamma\cap (K^\times\times\RR)$ because otherwise $F_G=\overline\Gamma$ and $K_G=K$ and this case is clear. 
Note that $\tau_x:=p(G)\cap \partial\sympol'$ is an element of $\cT$.
We can then identify
$$T_1=A_{\tau_x}\qquad \hbox{and}\qquad T_2=\Hom\left(\frac{M\cap K_G}{M \cap K_G^\times},S^1\right)$$
which gives the desired quotient representation of $g^{\overline X}_\log(\rho(\shM_{\overline X})^{-1}(x))$ as in Def.~\ref{def:skeleton5}.
\end{proof}

\subsection{Retraction}

\begin{theorem}[{\bf Main Theorem for General Cones}] 
The skeleton $S_{\sympol,\cT,K}$ embeds in $Z$ as a strong deformation retract.
\end{theorem}
\begin{proof}
By Thm~\ref{thm:embgen}, we have an embedding $j:S_{\sympol,\cT,K}\hra (\overline X(\shF_{\overline X})_{log})_1$ and by Thm~\ref{thm:fiberbdlgen} a homeomorphism $Z\cong (\overline X_{log})_1$.
It remains to show that $j$ is a strong deformation retraction. 
This works word by word the same way in the reasoning that led to Thm~\ref{thm:retract}.
The point is, that, from the proof of Theorem \ref{thm:embgen} above, we have
an explicit description of the fibers of the map $\rho({\shF_{\overline X}})$ over a point
$x\in \widehat{\partial\sympol'}$ and these are ``constant'' on the interiors of the simplices of $\cT$.  This allows us to use Lemma \ref{lem:slight-variant2} to
lift retractions and construct an iterative
argument exactly as in Theorem \ref{thm:retract}.
\end{proof}


\begin{thebibliography}{CK}


\bibitem[A06]{A} M. Abouzaid,
``Homogeneous Coordinate Rings and Mirror Symmetry for Toric Varieties,''
Geom. Topol. {\bf 10} (2006) 1097--1156

\bibitem[AAK]{AAK} M. Abouzaid, D. Auroux, L. Katzarkov: 
``Lagrangian fibrations on blowups of toric varieties and mirror symmetry for hypersurfaces'',  arXiv:1205.0053.


\bibitem[AF59]{AF} A. Andreotti and T. Frankel, ``The Lefschetz Theorem on
Hyperplane Sections,'' Ann. Math. {\bf 69} (1959) 713--717.

\bibitem[Ba93]{Ba93} Batyrev, V.~V.: 
``Variations of the mixed Hodge structure of affine hypersurfaces in algebraic tori'',
{\it Duke Math. J.,}
{\bf 69}(2), (1993)  349--409.



\bibitem[Be04]{Berkovich} Berkovich, V., ``Smooth p-adic analytic spaces are locally contractible. II.'' Geometric aspects of Dwork theory. Vol. I, II, (2004) 293--370. 

\bibitem[BM71]{BM} H. Bruggesser and P. Mani, ``Shellable decompositions of cells and spheres,''
Math. Scand. {\bf 29} (1971) 197--205.

\bibitem[CE]{CE} K. Cieliebak and Y. Eliashberg,
\emph{Symplectic Geometry of Stein Manifolds,} book manuscript in preparation.




\bibitem[DK87]{DK} V. I. Danilov and A. G. Khovanski\v{\i},
``Newton Polyhedra and an Algorithm for Computing Hodge-Deligne Numbers,''
Math. USSR Izvestiya {\bf 29} (1987) 279--298.

\bibitem[De82]{Deligne} P. Deligne, ``Hodge cycles on abelian varieties (notes by JS Milne),'' in Hodge cycles, motives, and Shimura varieties, Lecture Notes in Math. {\bf 900}, (1982) pp. 9--100

\bibitem[Di04]{Di} A. Dimca, ``Sheaves in Topology,'' Springer, 2004.


\bibitem[Fo57]{Fox} R. Fox, ``Covering spaces with singularities'' in A symposium in honor of S. Lefschetz, Princeton University Press (1957) 243--257

\bibitem[Fu93]{Fu}W. Fulton, {\em Introduction to toric varieties,}
Annals of Mathematics Studies {\bf 131},
Princeton University Press, (1993).

\bibitem[FLTZ11]{FLTZ} B. Fang, C.-C.M. Liu, D. Treumann and E. Zaslow,
``A Categorification of Morelli's Theorem,'' Invent. Math. {\bf 185} (2011);
{\tt arXiv:1007.0053}.

\bibitem[FLTZ]{FLTZ2} B. Fang, C.-C.M. Liu, D. Treumann and E. Zaslow,
``The Coherent-Constructible Correspondence for Toric Deligne-Mumford Stacks,''
{\tt arXiv:0911.4711} to appear in IMRN.


\bibitem[GS11]{GS2} M. Gross and B. Siebert,
``From Real Affine Geometry to Complex Geometry,''
Ann. of Math. {\bf 174} (2011) 1301--1428.

\bibitem[GKR]{GKR} M. Gross, L. Katzarkov, H. Ruddat: ``Towards Mirror Symmetry for Varieties of General Type'', arXiv:1202.4042.

\bibitem[GKZ94]{GKZ} I. Gelfand, M. Kapranov, and A. Zelevinsky, ``Discriminants, resultants, and multidimensional determinants'', Birkh\"auser, Boston, 1994

\bibitem[H02]{Hatcher} A Hatcher, \emph{Algebraic topology,}
Cambridge University Press, 2002.

\bibitem[K89]{Ka89} K. Kato, ``Logarithmic structures of Fontaine-Illusie", Algebraic Analysis, Geometry and Number Theory (Igusa, J.-I., ed.), Johns Hopkins University Press, Baltimore, 1989, pp.191-224.

\bibitem[K96]{Ka96} F. Kato, ``Log smooth deformation theory'',
Tohoku Math. J. (2) {\bf 48}(3), 1996, pp.317-354. 



\bibitem[KN99]{KN99} K. Kato, C. Nakayama: 
``Log Betti cohomology, log \'etale cohomology, and log de Rham cohomology of log schemes over $\CC$'', 
\emph{Kodai Math.\ J.}, {\bf 22}, 1999, p.161--186.

\bibitem[KS05]{KS05} M. Kontsevich, Y. Soibelman: 
``Affine structures and non-archimedean analytic spaces'',  
in "The Unity of Mathematics" in honor of the 90-th anniversary of I.M.Gelfand, Progress in Mathematics {\bf 244}, Birkh\"auser, (2005) 312--385.

\bibitem[K02]{Kozlov} D. Kozlov, \emph{Combinatorial algebraic topology,} Springer, 2002.

\bibitem[M00]{Mc} J. B. McCammond, ``A general small cancellation theory,'' Int. J. Algebra Comput., {\bf 10}, No. 1 (2000).

\bibitem[M04]{M} G. Mikhalkin, ``Decomposition into pairs-of-pants for complex
algebraic varieties,'' Topology {\bf 43} (2004) 1035--1065.

\bibitem[M59]{Mi} J. Milnor, ``On Spaces Having the Homotopy Type of a CW-Complex," Trans. Amer. Math. Soc. {\bf 90}, No. 2, (1959) 272--280. 


\bibitem[NO10]{NO10} C. Nakayama, A. Ogus: 
``Relative rounding in toric and logarithmic geometry,'', 
\emph{Geom. Topol.}, {\bf 14},  2010, p.2189--2241.

\bibitem[R10]{R} H.~Ruddat:
``Log Hodge groups on a toric Calabi-Yau degeneration'',
in Mirror Symmetry and Tropical Geometry, \emph{Contemporary Mathematics} {\bf 527}, Amer. Math. Soc., Providence, RI, 2010, p. 113-164.

\bibitem[STZ]{STZ} N. Sibilla, D. Treumann and E. Zaslow,
``Ribbon Graphs and Mirror Symmetry I,'' {\tt arXiv:1103.2462.}

\bibitem[SY09]{SY} J. Smrekar and A. Yamashita, ``Function Spaces of CW Homotopy Type Are Hilbert Manifolds," Proc. Amer. Math. Soc., {\bf 137}, No. 2, (2009) 751--759.



\bibitem[TZ]{TZ} D. Treumann and E. Zaslow,
``Polytopes and Skeleta,'' {\tt arXiv:1109.4430.}

\bibitem[W09]{Woolf} J. Woolf, ``The fundamental category of a stratified space,'' J. Homotopy Relat. Struct. {\bf 4} (2009) 359--387

\end{thebibliography}
\end{document}